\documentclass[11pt]{amsart}

\usepackage{amsmath,amsfonts,amssymb,amsthm,enumerate,tikz-cd}

\usepackage{hyperref} 
\hypersetup{
	colorlinks=true, 
	     urlcolor= blue,  
	     linkcolor= blue, 
     citecolor=blue,	  
	bookmarksopen=true,           
}
\include{macros}

\def\Z{\mathbb Z}
\def\R{\mathbb R}
\def\Q{\mathbb Q}
\def\F{\mathbb F}

\def\Aut{\mathrm{Aut}}

\theoremstyle{plain}
\newtheorem{theorem}{Theorem}
\newtheorem{conjecture}{Conjecture}
\newtheorem{lemma}{Lemma}
\newtheorem{proposition}{Proposition}

\theoremstyle{definition}

\newtheorem{remark}{Remark}

\newcommand{\Iso}{\mathrm{Iso}}

\DeclareMathOperator{\Gal}{Gal}

\newcommand{\sgn}{\mathrm{sgn}}
\newcommand{\CH}{\mathrm{CH}}

\newcommand{\Roots}{\mathrm{Roots}}

\newcommand{\Spec}{\mathrm{Spec}}
\newcommand{\Cl}{\mathrm{Cl}}

\newcommand{\ur}{\mathrm{ur}}

\newcommand{\dR}{\mathrm{dR}}

\newcommand{\Ind}{\mathrm{Ind}}

\newcommand{\Br}{\mathrm{Br}}
\newcommand{\Sym}{\mathrm{Sym}}

\newcommand{\sep}{\mathrm{sep}}

\newcommand{\Nm}{\mathrm{Nm}}
\newcommand{\Jac}{\mathrm{Jac}}
\newcommand{\et}{\mathrm{\acute{e}t}}
\newcommand{\loc}{\mathrm{loc}}

\DeclareMathOperator{\rk}{rk}
\newcommand{\cores}{\mathrm{cores}}

\newcommand{\NS}{\mathrm{NS}}

\newcommand{\Res}{\mathrm{Res}}

\newcommand{\val}{\mathrm{val}}

\newcommand{\Ker}{\mathrm{Ker}}

\newcommand{\Hom}{\mathrm{Hom}}
\newcommand{\End}{\mathrm{End}}

\begin{document}
\title{$2$-descent for Bloch--Kato Selmer groups and rational points on hyperelliptic curves II}
\author{Netan Dogra}
\maketitle
\pagestyle{headings}
\markright{$2$-DESCENT FOR BLOCH--KATO SELMER GROUPS II}

\begin{abstract}
We give refined methods for proving finiteness of the Chabauty--Coleman--Kim set $X(\mathbb{Q}_2 )_2 $, when $X$ is a hyperelliptic curve with a rational Weierstrass point. The main developments are methods for computing Selmer conditions at $2$ and $\infty$ for the mod $2$ Bloch--Kato Selmer group associated to the higher Chow group $\mathrm{CH}^2 (\mathrm{Jac}(X),1)$. As a result we show that most genus 2 curves in the LMFDB of Mordell--Weil rank 2 with exactly one rational Weierstrass point satsify $\# X(\mathbb{Q}_2 )_2 <\infty $. We also obtain a field-theoretic description of second descent on the Jacobian of a hyperelliptic curve (under some conditions).
\end{abstract}

\tableofcontents

\section{Introduction}
This paper is concerned with the question of extending the applicability of the Chabauty--Coleman--Kim method for studying rational points on higher genus curves $X/\Q $. This method produces sets $X(\Q _p )_n \subset X(\Q _p )$ for each $n>0$, which should provide successive approximations to $X(\Q )$:
\[
X(\Q _p )\supset X(\Q _p )_1 \supset X(\Q _p )_2 \supset \ldots \supset X(\Q ).
\]
Kim has conjectured \cite{BDCKW} that for $n\gg 0$, $X(\Q _p )_n =X(\Q )$. In the context we work in (that of genus two curves) it is often enough to produce a \textit{finite} set, since then the Mordell--Weil sieve can rule out `fake rational points'.

When $n=1$, the problem of finiteness of $X(\Q _p )_1$ is well understood, as this is exactly the set produced by the Chabauty--Coleman method: a sufficient condition is that $r<g$, where $r$ is the Mordell--Weil rank of the Jacobian of $X$ and $g$ is the genus of $X$. When $n=2$, it is known that $X(\Q _p )_2 $ is finite when $r<g+\rho (J)-1$, where $\rho (J):=\rk \NS (J)$ is the rank of the N\'eron--Severi group of $J$ (over $\Q $). 

The obstruction to extending this result is the rank of the mysterious $\Z _p $-module $H^1 _f (\Gal (\overline{\Q }|\Q ),\wedge ^2 T_p J)$, which is the Bloch--Kato Selmer group of the higher Chow group $\CH ^2 (\Jac (X),1)$, or of the Galois representation $H^2 _{\et }(\Jac (X)_{\overline{\Q }},\Z _p (2))\simeq \wedge ^2 T_p J$. The Bloch--Kato conjectures \cite{BK} give a precise formula for this number, which implies that $X(\Q _p )_2 $ should be finite whenever $r<g^2 +\rho (J)-1$. In this series of papers we introduce $2$-descent methods for proving bounds on $H^1 _f (\Gal (\overline{\Q }|\Q ),\wedge ^2 T_2 J)$. The first paper in this series \cite{BKdescent1} applied these to provably determine the set of rational points on the curve $y^2 -y=x^5-x$, answering a question of Bugeaud, Mignotte, Siksek, Stoll and Tengely \cite{BMSST}. In this paper we extend these methods to provide fairly robust criteria for finiteness of $X(\Q _2 )_2$ for a genus $2$ curve $X$ with a rational Weierstrass point. As an application we prove the following.
\begin{theorem}\label{thm:stats}
Of the 6,603 curves in the LMFDB with Mordell--Weil rank 2 and exactly one rational Weierstrass point, at least $3,323$ satisfy $\# X(\Q _2 )_2 <\infty $. Moreover, for each such $X$, $X'(\Q _2 )_2$ is finite whenever $X'$ is a quadratic twist of $X$ with Jacobian of Mordell--Weil rank $\leq 2$.
\end{theorem}
By contrast, only 8 of the the 6,603 curves in question satisfy the condition $\rho (J)>1$ usually associated with quadratic Chabauty.

\subsection{Relation to elliptic curve Chabauty}
A number of authors including Bruin \cite{bruin}, Flynn and Wetherell \cite{flynn-wetherell} and recently Hast \cite{hast} have studied rational points on genus two curves using a method called elliptic curve Chabauty. Specifically, given a genus 2 curve $X$ over $\Q $ with a rational Weierstrass point one can find a finite set of covers $f_{\alpha }:Y_{\alpha }\to X$ (all isomorphic over $\overline{\Q }$) such that $X(\Q )=\cup f_{\alpha }(Y_{\alpha })$ and such that $\Jac (Y_{\alpha })$ contains an isogeny factor isomorphic to the Weil restriction of an elliptic curve. This elliptic curve will be the Jacobian of the genus one curve defined by the quartic obtained by removing a root of a degree 5 polynomial defining $X$.

This method may also be thought of as a nonabelian version of the Chabauty--Coleman method, in the sense that the Tate module of the Jacobian of $Y_{\alpha }$ is a subquotient of the \'etale fundamental group of $X_{\overline{\Q }}$. In fact the analogy is stronger: as we explain in section \ref{sec:ecc}, the Galois cohomology of $\wedge ^2 J[2]$ and $\Jac (Y_{\alpha })[2]$ are closed related. Intriguingly, however, the Selmer conditions can be \textit{different}, meaning that one can instances where the $2$-descent methods described in this paper can prove finiteness of $X(\Q _2 )_2$, but $2$-descent methods for elliptic curve Chabauty are insufficient, and vice versa.

\subsection{Arithmetic statistics of Bloch--Kato Selmer groups}
Although this paper only considers the Bloch--Kato Selmer groups $H^1 _f (\Q ,\wedge ^2 T_2 J)$ for a hyperelliptic Jacobian $J$, it seems reasonable to believe that the general idea (bounding ranks of Bloch--Kato Selmer groups by `elementary' descent methods) has a larger realm of applicability. In considering the merits of such an approach, one may compare with the case of Selmer groups of elliptic curves. Here two notable applications of elementary $2$-descent are explicit computations on a given elliptic curve, and proving results in arithmetic statistics. This paper, and the previous paper in this series, demonstrate the analogue of the first application for the Selmer group of a higher Chow group. It would be interesting to further explore the analogue of the second application, i.e. to explore the extent to which it possible to prove results about Bloch--Kato Selmer groups `on average', or for a positive proportion of motives in a family. 

The simplest (new) case would seem to be the Galois representation $T_2 J(n)$, where $J$ is a $g$-dimensional hyperelliptic Jacobian and $n$ is a (nonzero) integer. Then the Bloch--Kato conjectures imply that the rank of $H^1 _f (\Q ,T_2 J(n))$ should be $g$ when $n>0$ and $0$ when $n<0$. In proposition \ref{prop:tail} we explain a relation between this conjecture and certain results of Ho, Shankar and Varma \cite{HSV}.

\subsection{Notation}
We will sometimes denote the Galois cohomology of a field $K$ with values in a $\Gal (K)$-module $M$ by $H^i (K,M)$. We will sometimes also use the notation $H^i (K,M)$ when $K=\prod K_i$ is merely an \'etale algebra with field factors $K_i$, and $M=(M_i )$ is a tuple of $\Gal (K)$-module $M_i$. In this case one can take $H^i (K,M)$ to denote the \'etale cohomology of $\Spec (K)$ with values in the corresponding sheaf on $\Spec (K)_{\et }$.

Given a finite extension $L|K$ of \'etale algebras, corresponding to a product $L_{ij}|K_i$ of field extensions, and a tuple $(M_{ij})$ of $\Gal (L_{ij})$-modules, we write $\Ind ^L _K M$ to mean the tuple $(\Ind ^{\Gal (K_i )}_{\Gal (L_{ij})}M_{ij})_i$ of $\Gal (K_i )$-modules.

\subsection*{Acknowledgements}
I am grateful to Lee Berry for many helpful discussions about boundary maps. This research was supported by a Royal Society University Research Fellowship.
\section{A review of $2$-descent for Bloch--Kato Selmer groups}\label{sec:recall}
We first (very briefly) recall the context of this paper. More details can be found in \cite{BKdescent1} (although there are usually better references for what we describe below, this reference will usually prove convenient because this paper is a continuation of it). Let $X/\Q $ be a smooth projective geometrically irreducible curve of genus $g$ with Jacobian $J$. We have the following sufficient condition for finiteness of $X(\Q _p )_2$.
\begin{lemma}[\cite{BKdescent1}, Lemma 25]\label{lemma:BK1_2}
The set $X(\Q _p )_2 $ is finite whenever
\[
\dim H^1 _f (G_{\Q },\wedge ^2 V_p J)<\frac{(3g-2)(g+1)}{2}-\rk J(\Q ).
\]
\end{lemma}
The goal of this paper is to find ways for bounding the dimension of $H^1 _f (G_{\Q},\wedge ^2 V_2 J)$. The dimension of $H^1 _f (G_{\Q },\wedge ^2 V_2 J)$ is predicted by the following special case of the Bloch--Kato conjectures (see \cite[\S 6]{BKdescent1} for an explanation of the deduction).
\begin{conjecture}\label{conj:BK}[\cite{BK}]
Let $Z/\Q $ be a smooth projective variety with good reduction outside $S$, suppose $2j\geq i$. Then
\[
\dim H^1 _f (G_{\Q },H^j (Z_{\overline{\Q }},\Q _p (i)))=0.
\]
\end{conjecture}
The dimension of $H^1 _f (G_{\Q },\wedge ^2 V_2 J)$ is (unconditionally) equal, via the short exact sequence
\[
0\to \wedge ^2 T_2 J\to \wedge ^2 V_2 J\to \wedge ^2 V_2 J/\wedge ^2 T_2 J\to 0,
\]
to the rank of $H^1 _f (G_{\Q },\wedge ^2 T_2 J)$. 

In this paper, unlike \cite{BKdescent1}, we do not attempt to use the crystalline condition. Instead, we will try to estimate the rank of the larger space $H^1 _{f,\{ 2\}}(G_{\Q },\wedge ^2 T_2 J)$. By definition this is simply the subspace of cohomology classes which are unramified outside a finite set of primes and whose restriction to each prime away from $2$ is torsion. We will sometimes denote this subspace simply by $H^1 _{\{ 2\}}(G_{\Q },\wedge ^2 T_2 J)$. 

We first explain how much information working with $H^1 _{f,\{ 2\}}(G_{\Q },\wedge ^2 T_2 J)$.  throws away. We have
\[
\dim H^1 _f (G_{\Q },\wedge ^2 T_2 J)\leq \dim H^1 _{f,\{ 2\} }(G_{\Q },\wedge ^2 T_2 J)+1.
\]
In general, we would expect this inequality to be strict, as explained by the following lemma, which follows straightforwardly from standard facts about Galois cohomology and $p$-adic Hodge theory (see e.g. \cite{berger}).
\begin{lemma}\label{lemma:BK1}
We have
\[
\dim H^1 (G_{\Q _p },\wedge ^2 V_p J)-\dim H^1 _f (G_{\Q _p },\wedge ^2 V_p J)=\binom{g}{2}+\dim _{\Q _p } \Hom (\Q _p (1),\wedge ^2 V_p J).
\]
\end{lemma}

Hence Conjecture \ref{conj:BK} predicts that, replacing $H^1 _f $ with $H^1 _{f,\{2\}}$, we should be able to prove finiteness of $X(\Q _2 )_2$ whenever
\[
\rk J(\Q )<\binom{g+1}{2}.
\]
For example, when $g=2$, we expect $X(\Q _p )_2$ to be finite whenever the Mordell--Weil rank of the Jacobian is less than $4$, but we expect to be able to prove finiteness without using the crystalline condition whenever the rank is less than $3$.

The basic strategy for bounding the rank of $H^1 _{f}(G_{\Q },\wedge ^2 T_2 J)$, or $H^1 _{f,\{ 2\}}(G_{\Q },\wedge ^2 T_2 J)$, is to try to describe their image in $H^1 (G_{\Q },\wedge ^2 J[2])$. Recall that by the short exact sequence
\[
0\to \wedge ^2 T_2 J\stackrel{\cdot 2}{\longrightarrow }\wedge ^2 T_2 J\to \wedge ^2 J[2]\to 0
\]
we have a short exact sequence
\[
0\to H^1 (K,\wedge ^2 T_2 J)\otimes \F _2 \to H^1 (K,\wedge ^2 J[2])\to H^2 (K,\wedge ^2 T_2 J)[2]\to 0.
\]
Hence our goal is to find a subvector space of $H^1 (\Q ,\wedge ^2 T_2 J)\otimes \F _2 $ of dimension at least that of $H^1 _f (\Q ,\wedge ^2 V_2 J)$, and explicitly bound the dimension of its image in $H^1 (\Q ,\wedge ^2 J[2])$.

We now recall some constructions from \cite{BKdescent1} which sometimes achieve this. The starting point is an explicit field-theoretic description of $H^1 (\Q ,\wedge ^2 J[2])$. Let $K$ be a field of characteristic different from $2$, and let $X$ be a hyperelliptic curve defined by a separable polynomial $f\in K[x]$ of odd degree. Let $K_f $ denote the \'etale algebra $K[x]/(f)$. Let $K_{f,2}$ denote the \'etale algebra $K [x,y,\frac{1}{x-y}]/(f(x),f(y))$. Let $K_f ^{(2)}\subset K _{f,2}$ denote the subalgebra fixed by the involution swapping $\overline{x}$ and $\overline{y}$. Let 
\[
\Nm :K_f ^{(2)}\to K_f 
\]
be the composite of the inclusion $K_f ^{(2)}\to K_{f,2}$ with the norm map $K_{f,2}\to K_f $.
\begin{proposition}[\cite{BKdescent1}, Proposition 1]\label{prop:BKno1}
We have an isomorphism
\[
H^1 (K,\wedge ^2 J[2])\simeq \Ker (K_f ^{(2),\times }\otimes \F _2 \stackrel{\Nm }{\longrightarrow }K_f ^\times \otimes \F _2 ).
\]
\end{proposition}

\begin{proposition}[\cite{BKdescent1}]\label{prop:storysofar}
Suppose that $X$ has semistable reduction outside $S\cup \{2\}$, where $S$ is a finite set of odd primes.
\begin{enumerate}
\item Under the isomorphism from Lemma \ref{lemma:BK1}, the image of $H^1 _{\{ 2\}}(G_{\Q },\wedge ^2 T_2 J)$ in $\Ker (\Q _f ^{(2), \times}\otimes \F _2 \to \Q _f ^{\times }\otimes \F _2 )$ is contained in 
\[
M:=\Ker (\Q _f ^{(2),\times}\otimes \F _2 \stackrel{\val }{\longrightarrow } \oplus _{v\notin S}\F _2 ^{\# \Spec (\Q _{v,f}^{(2)})}).
\]
Here $\val $ is the valuation map sending an element of $\Q _f ^{(2),\times }\otimes \F _2 $ to its valuation mod $2$ under all embeddings from a field factor of $\Q _f $ to a finite extension of $\Q _v $ for a prime $v$ not in $S\cup \{ 2\}$. 
\item If 
\[
\dim _{\F _2 }\Cl (\mathcal{O}_{\Q _f ^{(2)}})[2] =\dim _{\F _2 }\Cl (\mathcal{O}_{\Q _f ^{(2)}})[2]
\]
then the image of $H^1 _{\{ 2\} } (\Q ,\wedge ^2 T_2 J)$ in $M$ is contained in the subspace of 
\[
\Ker (\mathcal{O}_{\Q _f ^{(2)}}\left[\frac{1}{2\prod _{v\in S }v}\right]^\times \otimes \F _2 \to \Q _f ^\times \otimes \F _2 
\]
whose image in $\R _f ^{(2),\times }\otimes \F _2 \oplus \Q _{2,f}^{(2),\times }\otimes \F _2 $ is in the image of $H^1 (\Q _2 ,\wedge ^2 T_2 J)\oplus H^1 (\R ,\wedge ^2 T_2 J)$.
\end{enumerate}
\end{proposition}
\begin{proof}
We briefly explain how to deduce this statement from results in \cite{BKdescent1}. As explained in \cite[Lemma 19]{BKdescent1}, the condition that $X$ has semistable reduction at $v\neq 2 $ allows us to ensure $H^1 _{f,\{2\}}(\Q ,\wedge ^2 T_2 J)\otimes \F _2 $ maps into 
\[
H^1 _{\ur }(\Q _v ,\wedge ^2 J[2])=\Ker (\Q _f ^{(2),\times }\otimes \F _2 \to \Q _f ^\times \otimes \F _2 \oplus \oplus _{w\in \Spec (\Q _f ^{(2)})}\F _2 ),
\]
where the second map is given by valuation. The condition on the class group allows us to ensure that $M$ can be identified with a subspace of $\Ker (\mathcal{O}_{\Q _f ^{(2)}}\left[\frac{1}{2\prod _{v\in S }v}\right]^\times \otimes \F _2 $ by a standard Galois cohomological argument \cite[Proposition 12.6]{PS97} (see \cite[Lemma 26]{BKdescent1} for details of the deduction).
\end{proof}
\section{Boundary maps for $J[4]$}\label{sec:boundary}
The main result in this paper is a method for computing appropriate `Selmer conditions' on $H^1 (\Q ,\wedge ^2 J[2])$ at the prime 2. We have two approaches to describing the image of $H^1 _f (G_{\Q _2 },\wedge ^2 T_2 J)$ in $H^1 (G_{\Q _2 },\wedge ^2 J[2])$. The first, which is needed in \cite{BKdescent1}, is to construct classes using the nonabelian $(x-T)$ map. Since it is easy to compute the rank of $H^1 _f (G_{\Q _2 },\wedge ^2 T_2 J)$, if one can find enough crystalline classes one can verify that one has found a basis for the image of $H^1 _f (G_{\Q _2 },\wedge ^2 T_2 J)$ in $H^1 (G_{\Q _2 },\wedge ^2 J[2])$ by a dimension count. 

The second approach, taken in this paper, is to construct liftability obstructions.  Specifically, in the next section we try to compute the boundary map
\[
H^1 (K,\wedge ^2 J[2])\to H^2 (K,\wedge ^2 J[2])
\]
associated to the short exact sequence
\[
0\to \wedge ^2 J[2]\to \wedge ^2 J[4]\to \wedge ^2 J[2]\to 0.
\]
In this section, as a step towards this we describe the boundary map
\[
\delta :H^1 (K,J[2])\to H^2 (K,J[2])
\]
associated to the short exact sequence
\[
0\to J[2]\to J[4]\to J[2]\to 0.
\]
The split short exact sequence
\[
0\to J[2]\to \Ind ^{K_f }_K \F _2 \to \F _2 \to 0
\]
gives isomorphisms
\[
H^1 (K,J[2])\simeq \Ker (K_f ^\times \otimes \F _2 \stackrel{\Nm }{\longrightarrow }K^\times \otimes \F _2 )
\]
and 
\[
H^2 (K,J[2])\simeq \Ker (H^2 (K_f ,\F _2 )\stackrel{\Nm }{\longrightarrow } H^2 (K,\F _2 )).
\]
Our characteristation of $\delta $ is in terms of these isomorphisms. We let $i_1 $ and $i_2 $ denote the two inclusions of $K_f $ into $K_{f,2}$, letting $\alpha $ and $\beta $ denote the images of $x$ under the corresponding maps
\[
K[x]\to K_{f,2}.
\]
Let $\cores ^{(1)}$ and $\cores ^{(2)}$ denote the corestriction maps corresponding to the inclusions of $K_f$ into $K_{f,2}$ via $i_1$ and $i_2$ respectively.
\begin{proposition}\label{prop:everything}
Suppose that $X$ descends to a subfield $K_0$ for which the Galois group of $K_0$ acts $2$-transitively on the set of roots of $f$ and for which $X$ has a $K_0$-rational non-Weierstrass point. Then the boundary map
\[
H^1 (K,J[2])\to H^2 (K,J[2])
\]
associated to $J[4]$ is given on $z\in H^1 (K,J[2])\subset K_f ^\times \otimes \F _2 $ by
\[
z\mapsto \cores ^{(2)} _{K_{f,2}|K_f}(i_1 (z)\cup ((\beta -\alpha )f'(\beta )))
\]
\end{proposition}
\begin{remark}
We expect that the formula for the boundary map should still hold without the condition on the existence of a rational non-Weierstrass point or the condition on the $2$-transitivity of the Galois group of $f$. The condition of a rational non-Weierstrass point is innocuous for our applications, and enables us to quote certain explicit models for abelian covers of $X_{\overline{K}}$ from \cite{BKdescent1}. In loc. cit. it was important that these were pointed $K$-covers lying over $(X,b)$ where $b$ was a rational non-Weierstrass point. The condition on $2$-transitivity of the Galois group is more restrictive, and allows us to apply Shapiro's lemma in a very direct way when describing boundary maps.
\end{remark}
In general, the computation of the boundary map involves working over a large \'etale algebra $K_f \otimes K_f $. However, for local fields these computations can actually be reduced to $K_f $, by local class field theory.
\begin{lemma}\label{lemma:local_norms}
Let $K$ be a finite extension of $\Q _p $. Let $L|K$ be a finite extension. Then the norm map
\[
\cores :H^2 (K,\Ind ^L _K \mu _2 )\to H^2 (K,\mu _2 )
\]
is an isomorphism of one dimensional $\F _2 $-vector spaces.
\end{lemma}

Let $f=c\cdot \prod _{i=1}^m f_i $ be a factorisation of $f$ into monic irreducibles, and $c\in K$. Let $K_i :=K[x]/(f_i )$, and let $\alpha _i $ be the image of $x$ in $K_i $. Finally, let 
\[
\langle ,\rangle _{K_i }:K_i ^\times \otimes \F _2 \times K_i \times \otimes \F _2 \to \F _2 \simeq H^2 (K_i ,\F _2 )
\]
be the Hilbert symbol map. Note that we have an isomorphism
\[
H^2 (K,J[2])\simeq \Ker (\F _2 ^m \to \F _2 )
\]
by Tate duality. Hence the boundary map may be viewed as a map
\[
\Ker (\oplus _{i=1}^m K_i ^\times \otimes \F _2 \stackrel{\Nm }{\longrightarrow }K^ \times \otimes \F _2 )\to \F _2 ^m .
\]
\begin{lemma}\label{lemma:LCFT_boundary}
Suppose $X$ is defined over a field $K_0$, and that the Galois group of $K_0$ acts $2$-transitively on the roots of $f$ and $X$ has a $K_0 $-rational non-Weierstrass point.
If $K$ is a finite extension of $\Q _p $ and an extension of $K_0$, then the boundary map
\[
H^1 (K,J[2])\to H^2 (K,J[2])\simeq (\Ker (\oplus _i \F _2 \to \F _2 ))^*
\]
is given by sending a tuple $(z_i )\in \prod _i K_i ^\times \otimes \F _2 $ to the class of the tuple $(w_i )\in \F _2 ^m $, where 
\[
w_i :=\langle -\prod _{j\neq i}f_j (\alpha _i ),z_i \rangle _{K_i } +\sum _{j\neq i}\langle cf_i (\alpha _j ),z_j \rangle _{K_j }
\]
\end{lemma}
\begin{proof} 
By Proposition \ref{prop:everything}, it is enough to show that $w_i$ is equal to the $H^2 (K_i ,\mu _2 )$ component of 
\[
\cores ^{(2)}_{K_{f,2}|K_f} ((z_j ),(\beta -\alpha )f'(\beta )).
\]
We have
\[
\cores ^{(2)}_{K_{f,2}|K_f} ((z_j ),-f'(\beta )) =(\Nm ^{(2)}_{K_{f,2}|K_f }(i_1 (z_j ))\cup -f'(\beta ))
\]
by the compatibility of cup products with corestrictions. For any $(z)\in H^1 (K_f ,\mu _2 )$, we have
\[
i(\Nm _{K_f |K}(z))=\Nm ^{(2)}_{K_{f,2}|K_f}(i_1 (z))+i(z)
\]
using the decomposition $K_f \otimes K_f \simeq K_{f,2}\times K_f $. In particular, if $z$ is in $\Ker (\Nm )$, we obtain
\[
\Nm ^{(2)}_{K_{f,2}|K_f }(i_1 (z))=z, 
\]
hence for $(z_j )$ as above
\[
\cores ^{(2)}_{K_{f,2}|K_f} ((z_j ),-f'(\beta ))=( \langle z_i ,-f'(\alpha _i )\rangle _{K_i })_i 
\]
To compute $\cores ^{(2)}_{K_{f,2}|K_f }((z_j )\cup (\alpha -\beta ))$, we decompose $K_{f,2}$ into fields $K_{ijk}$ lying above $K_i $ and $K_j$, then we have
\[
\sum _{j,k}\cores _{K_{jik}|K_i }(z_j \cup (\alpha _j -\beta _i )),
\]
which by Lemma \ref{lemma:local_norms} is equal to 
\[
\sum _{j,k}\cores _{K_{jik}|K_i}(z_j \cup (\alpha _j -\beta _i )).
\]
Applying Lemma \ref{lemma:local_norms} again, we see that this is equal to
\[
\sum _{j,k}\cores _{K_{jik}|K}(z_j \cup (\alpha _j -\beta _i )).
\]
We have
\[
\sum _{k}\cores _{K_{jik}|K_j }(z_j \cup (\alpha _j -\beta _i ))=\begin{cases}\langle z_j ,f_i (\alpha _j )\rangle _{K_j }, & i\neq j \\ \langle z_j ,f_i ' (\alpha _i ) \rangle _{K_i } , & i=j. \end{cases}.
\]
Hence the lemma follows from noting that $f'(\alpha _i )=cf_i ' (\alpha _i )\prod _{j\neq i}f_j (\alpha _i )$.
\end{proof}
\subsection{Strategy of proof of Proposition \ref{prop:everything}}
We will reduce the proof of Proposition \ref{prop:everything} to certain explicit calculations with $H^1 (K,\End (J[2]))$, via the following well known characterisation of boundary maps.
\begin{lemma}\label{lemma:boundary_cup}
Let $G$ be a group, $M$ and $N$ finite $G$-modules, and let $E$ be an extension of $M$ by $N$. Let $E^c$ be a twist of this extension by $c\in H^1 (G,\Hom (M,N))$. Let $\delta $ and $\delta _c$ be the boundary maps $H^1 (G,M)\to H^2 (G,N)$ associated to the extensions $E$ and $E^c$ respectively. Then, for all $\alpha \in H^1 (G,M)$,
\[
\delta _c (\alpha )=\delta (\alpha )+c\cup \alpha .
\]
\end{lemma}
\begin{proof}
Let $\rho :G\to \Aut (E)$ and $\rho _c :G\to \Aut (E^c )$ denote the $G$-actions on $E$ and $E^c$ respectively. Let $\alpha \in Z^1 (G,M)$. Choose a section $s$ of $E\to M$, which we may also think of as a section of $E^c \to M$. Then
\[
\delta (\alpha )(g_1 ,g_2 )=\widetilde{\alpha }(g_1 g_2 )-\rho (g_1 )\cdot \widetilde{\alpha }(g_2 )-\widetilde{\alpha }(g_1 ),
\]
and similarly
\[
\delta _c (\alpha )(g_1 ,g_2 )=\widetilde{\alpha }(g_1 g_2 )-\rho _c (g_1 )\cdot \widetilde{\alpha }(g_2 )-\widetilde{\alpha }(g_1 ).
\]
Hence 
\[
\delta _c (\alpha )(g_1 ,g_2 )-\delta (g_1 ,g_2 )=(\rho (g_1 )-\rho _c (g_1 ))\cdot \widetilde{\alpha }(g_2 ).
\]
On the other hand, for any $e\in E$ and $g\in G$, $\rho _c (g)(e)=\rho (g)(e)+c(g)(g\cdot (e))$, from which Lemma \ref{lemma:boundary_cup} follows.
\end{proof}
In particular, if $\delta (\alpha )=0$, we can describe the $\delta _c$ purely in terms of the cup product. If $E_2 $ is the twist of $E_1$ by $c$, we shall sometimes denote the class of $c$ in $H^1 (K,\Hom (M,N))$ by $[E_2 ]-[E_1 ]$. We note that $c$ admits the following description.
\begin{lemma}
If $E_1 $ and $E_2$ are extensions of $M$ by $N$ which are isomorphic as abelian groups, then $E_2$ is the twist of $E_1 $ by the $G$-equivariant $\Hom (M,N)$-torsor of isomorphisms $E_1 \simeq E_2$ of extensions of $M$ by $N$.
\end{lemma}

It remains to a give an explicit description of the cup product map. To apply Lemma \ref{lemma:boundary_cup}, we need to identify a self-extension of $J[2]$ for which the boundary map is trivial. The norm map
\[
N :\Ind ^{K_f }_K \mu _4 \to \mu _4
\]
is split by plus or minus the inclusion $\mu _4 \to \Ind ^{K_f }_K \mu _4 $, depending on the parity of $g$. It follows that $\Ker (N )$ is an extension of $J[2]$ by $J[2]$ for which the boundary map $\delta _{\Ker (N)}$ vanishes.

\begin{proposition}\label{prop:another_explicit_description}
Suppose $f$ is defined over a subfield $K_0$ of $K$ such that $\Gal (K_0 )$ acts $2$-transitively on the roots of $f$.
\begin{enumerate}
\item We have an isomorphism
\[
H^1 (K,\Hom (J[2],\Ind ^{K_f }_K \F _2 ))\simeq K_{f,2}^\times \otimes \F _2 . \]
With respect to this isomorphism, the cup product map
\[
H^1 (K,J[2])\times H^1 (K,\Hom (J[2],\Ind ^{K_f }_K \F _2 )\to H^2 (K,\Ind ^{K_f }_K \F _2 )
\]
is given by
\[
(g(\alpha ),h(\alpha ,\beta ))\mapsto \cores ^{(2)}_{K_{f,2}|K_f }(i_1 (g(\alpha ))\cup h(\alpha ,\beta )).
\]
\item Suppose that $X$ has a rational non-Weierstrass point over $K_0$. Then with respect to this isomorphism, the class of the torsor of isomorphisms between $J[4]$ and $\Ker (N)$ is equal to $(\alpha -\beta )f'(\beta ).$
\end{enumerate}
\end{proposition}

\subsection{Field-theoretic description of $H^1 (K,\End (J[2]))$ and the cup product map}
We will prove Proposition \ref{prop:another_explicit_description} in several stages. In this subsection we discuss the proof of part (1). We first recall the explicit form of Shapiro's lemma, which we will make repeated use of. Let $H<G$ be a finite index subgroup, and let $M$ be a finite $H$-module. Let $\pi :\Res ^G _H \Ind ^G _H M\to M$ be the map induced by adjunction.
\begin{lemma}\label{lemma:concrete_shapiro}
The composite map
\[
H^1 (G,\Ind ^G _H M) \stackrel{\Res }{\longrightarrow }H^1 (H,\Res ^G _H \Ind ^G _H M) \stackrel{\pi _* }{\longrightarrow }H^1 (H,M) 
\]
is equal to the isomorphism from Shapiro's lemma.
\end{lemma}
\begin{proof}
See \cite{stix2010trading}, \cite{stix2013correction}, or \cite{NSW} for more general statements (for nonabelian cohomology and cohomology in arbitrary degree respectively).
\end{proof}

For a root $\gamma $ of $f$, let $D_{\gamma }\subset J[2]$ denote the kernel of the norm map
\[
\F _2 [\Roots (f)-\{ \gamma \} ]\to \F _2 .
\]
\begin{lemma}\label{lemma:the_End}
Suppose $\Gal (K^{\sep }|K)$ acts transitively on the roots of $f$. Then we have an isomorphism
\begin{equation}\label{eqn:End_Iso}
H^1 (K,\End (J[2]))\simeq \Ker(H^1 (K_f ,\Hom (J[2],\mu _2 ))\to H^1 (K,\Hom (J[2],\mu _2 )))
\end{equation}
from Shapiro's lemma, which is given by the composite map
\[
H^1 (K,\End (J[2]))\to H^1 (K_f ,\End (J[2]))\to H^1 (K_f ,\Hom (J[2],\mu _2 ))
\]
where the first map is restriction and the second map is the projection induced by quotienting $J[2]$ by $D_{\alpha }$.

\end{lemma}
\begin{proof}
Since the exact sequence 
\begin{equation}\label{eqn:split}
0\to J[2]\to \Ind ^{K_f }_K \F _2 \to \F _2 \to 0
\end{equation}
splits, we have an isomorphism
\[
H^1 (K,\End (J[2]))\simeq \Ker (H^1 (\Hom (J[2],\Ind ^{K_f }_K \F _2 ))\to H^1 (K,\Hom (J[2],\F _2 ))).
\]
Via the $\Gal (K^{\sep }|K)$-equivariant isomorphism 
\[
\Hom (J[2],\Ind ^{K_f }_K \F _2 )\simeq \Ind ^{K_f }_K \Hom (J[2],\F _2 ),
\] 
together with Shapiro's lemma, we obtain \eqref{eqn:End_Iso}.
\end{proof}

If we restrict to $K_f $, we obtain a $\Hom (J[2],\F _2 )$-torsor of unipotent isomorphisms
\[
J[4]/2D_{\alpha }\simeq \Ker (\Z /4\Z [\Roots (f)]\to \Z /4\Z )/2D_{\alpha }
\]
Suppose $\Gal (f)$ acts $2$-transitively on the roots of $f$. Then we have an isomorphism
\[
\Ind ^H _{H^{(2)}}\F _2 \simeq J[2]
\]
induced by the isomorphism
\[
\F _2 [\Roots (f(x)/(x-\alpha ))]\simeq \Ker (\F _2 [\Roots (f)]\to \F _2 )
\]
given by sending $[\gamma ]$ to $[\gamma ] -[\alpha ]$.

\begin{lemma}\label{lemma:another_lemma}
We have a commutative diagram
\begin{equation}\label{eqn:shapiro_again}
\begin{tikzcd}
H^1 (K_f ,\Hom (J[2],\mu _2 )) \arrow[r] \arrow[d] & H^1 (K,\Hom (J[2],\mu _2 )) \arrow[d] \\
H^1 (K_{f,2},\F _2 ) \arrow[r]           & H^1 (K_f ,\F _2 )/H^1 (K,\F _2 )          
\end{tikzcd}
\end{equation}
whose vertical maps are isomorphisms, where the top horizontal map is corestriction, and the bottom horizontal map is the composite of 
\[
\Nm ^{(1)}_{K_{f,2}|K_f }:K_{f,2}^\times \otimes \F _2 \to K_f ^\times \otimes \F _2
\]
with the projection $K_{f}^\times \otimes \F _2 \to H^1 (K_f ,\F _2 )/H^1 (K,\F _2 )$.
\end{lemma}

We deduce the following.
\begin{lemma}\label{lemma:another_shapiro}
Suppose $\Gal (K)$ acts $2$-transitively on the roots of $f$. Then we have an isomorphism
\[
H^1 (K,\End (J[2]))\simeq \Ker (H^1 (K_{f,2},\F _2 )\stackrel{\Nm ^{(1)}}{\longrightarrow } H^1 (K_f ,\F _2 )/H^1 (K,\F _2 ))
\]
given by the composite of the isomorphism from Lemma \ref{lemma:the_End} with the map
\begin{align*}
& \Ker (H^1 (K_f ,\Hom (J[2],\mu _2 )) \to & H^1 (K,\Hom (J[2],\mu _2 )) 
\\ \Ker (H^1 (K_{f,2},\F _2 ) \to           & H^1 (K_f ,\F _2 )/H^1 (K,\F _2 )          )
\end{align*}
from Lemma \ref{lemma:another_lemma}.
\end{lemma}
\begin{proof}
This again comes from Shapiro's lemma and the direct sum decomposition $\Ind ^{K_f }_K \F _2 \simeq J[2]\oplus \F _2$.
\end{proof}
Since we are adopting the convention of describing the boundary map for $J[4]$ as a map from norm one elements of $K(\alpha )$ to $2$-torsion in the Brauer group of $K(\beta )$, it will be convenient to swap the roles of $\alpha $ and $\beta $ in the above discussion (i.e. to use the involution generating $\Aut (K_{f,2}|K_f ^{(2)})$) giving an isomorphism
\begin{equation}\label{eqn:swap}
H^1 (K,\End (J[2])) \simeq \Ker (H^1 (K_{f,2},\F _2 )\stackrel{\Nm ^{(2)}}{\longrightarrow } H^1 (K_f ,\F _2 )/H^1 (K,\F _2 )).
\end{equation}

It remains to check the second claim in part (1) of Proposition \ref{prop:another_explicit_description}, which specifies how the cup product map
\[
H^1 (K,J[2])\times H^1 (K,\End (J[2]))\to H^2 (K,J[2])
\]
relates to the field theoretic description of the cohomology groups involved. This amounts to describing how Shapiro's lemma behaves with respect to cup products.

\begin{lemma}\label{lemma:compatibility_with_cup}
Let $H,G,M$ be as in Lemma \ref{lemma:concrete_shapiro}, and $N$ a $G$-module. Then the diagram
\[
\begin{tikzcd}
H^1 (G,\Ind ^G _H M)\otimes H^1 (G,N) \arrow[r] \arrow[d] & H^2 (G,\Ind ^G _H (M)\otimes N) \arrow[d] \\
H^1 (H,M)\otimes H^1 (H,N) \arrow[r] & H^2 (H,M\otimes N) \\
\end{tikzcd}
\]
commutes, where the right vertical map is the composite of Shapiro's isomorphism with the isomorphism $\Ind ^G _H (M)\otimes N\simeq \Ind ^G _H (M\otimes N)$.
\end{lemma}
\begin{proof}
By Lemma \ref{lemma:concrete_shapiro}, this reduces to commutativity of
\[
\begin{tikzcd}
H^1 (G,\Ind ^G _H M)\otimes H^1 (G,N) \arrow[r] \arrow[d] & H^2 (G,\Ind ^G _H (M)\otimes N) \arrow[d] \\
H^1 (H,\Res ^G _H \Ind ^G _H M)\otimes H^1 (H,N) \arrow[r] \arrow[d] & H^2 (H,\Res ^G _H \Ind ^G _H (M)\otimes N) \arrow[d] \\
H^1 (H,M)\otimes H^1 (H,N) \arrow[r] & H^2 (H,M\otimes N). \\
\end{tikzcd}
\]
commutativity of the top part is just compatibility of cup products with restriction maps. Commutativity of the second part amounts to the fact that the projection map
\[
\Res ^G _H \Ind ^G _H (M\otimes N)\to M\otimes N
\]
is equal to the tensor product of the projection map on $M$ with the identity on $N$.
\end{proof}

\begin{proposition}\label{prop:cup_end}
Let $H<G$ be a finite index subgroup and let $R$ be a ring with trivial $G$-action (and discrete topology). Let $M$ be a discrete $R[G]$-module. 
Let $\theta \in H^1 (H,M)$ and  $\psi \in H^1 (H,R)$ map to $\widetilde{\theta }$ and $\widetilde{\psi }$ in $H^1 (G,\Ind ^G _H M)$ and $H^1 (G,\Ind ^G _H R)$ under the isomorphism of Shapiro's lemma. Then we have an equality
\[
\widetilde{\theta }\cup \widetilde{\psi }=\cores (\theta \cup \psi )
\]
of classes in $H^2 (G,M)$.
\end{proposition}
\begin{proof}
In the absence of a conceptual proof, we write out cochains. Let $g_1 ,\ldots ,g_m$ be a set of coset representatives for $G$, and define $\gamma :G\to H$ and $n:G\to \{1,\ldots ,m\}$ by the property that $g=\gamma (g)\cdot g_{n(g)}$. 

Choose coset representatives for $\psi $ and $\theta $, which we will also denote by $\psi $ and $\theta $. Since $G$ and $H$ act trivially on $R$, the cup product of $\theta $ and $\psi $ is given on cocycles by
\[
(\theta \cup \psi )(x_1 ,x_2 )= \theta (x_1 )\cdot \psi (x_2 ).
\]
Hence the corestriction of $\theta \cup \psi $ is given by 
\[
(x_1 ,x_2 )\mapsto \sum _i g_i ^{-1}\theta (\gamma (g_i )^{-1}\gamma (g_i x_1 ))\otimes \psi (\gamma (g_i x)^{-1}\gamma (g_i x_1 x_2 )).
\]
The cup product of $\widetilde{\theta }$ and $\psi $ is the $2$-cocycle
\[
G\times G\to \Ind ^G _H M\otimes \Ind ^G _H R
\]
given by 
\[
(x_1 ,x_2 )\mapsto \widetilde{\theta }(x_1 )\otimes x_1 \widetilde{\psi }(x_2 ).
\]
We have 
\[
\widetilde{\theta }(x_1 )=\sum _i \theta (u(g_i x_1 ))\cdot [g_i ]
\]
and 
\[
\widetilde{\psi }(x_2 )=\sum _i \psi (u(g_i x_2 ))\cdot [g_i ].
\]
When we project $\widetilde{\theta }\cup \widetilde{\psi }$ to $H^2 (G,M)$, we see that it is given by
\[
(x_1 ,x_2 )\mapsto \sum _i \theta (\gamma (g_i x_1 ))\cdot \psi (\gamma (g_{m_i }x_2 ))
\]
where $m_i$ has the property that $Hg_i x_1 =Hg_{m_i }$. The proposition now follows from the identity
\[
\gamma (g_i x_1 x_2 )=\gamma (g_i x)\gamma (g_{m_i }y).
\]
\end{proof}
To complete the proof of part (1) of Proposition \ref{prop:another_explicit_description}, it is enough to show that the diagram
\[
\begin{tikzcd}
H^1 (K,J[2])\times H^1 (K,\Hom (J[2],\Ind ^{K_f }_K \F _2 ))  \arrow[r] \arrow[d] & H^2 (K,\Ind ^{K_f }_K \F _2 ) \arrow[d] \\
\Ker (K_f ^\times \otimes \F _2 \stackrel{\Nm }{\longrightarrow }K^\times \otimes \F _2 )\times K_{f,2}^\times  \otimes \F _2 \arrow[r] & H^2 (K_f ,\F _2 ) \\
\end{tikzcd}
\]
commutes, where the vertical maps are the isomorphisms above, the top horizontal map is the cup product and the bottom horizontal map is
\[
(x,y)\mapsto \cores ^{(2)}_{K_{f,2}|K_f }(i_1 (x)\cup y).
\]
Via the split exact sequence \eqref{eqn:split}, it is enough to prove that the diagram 
\[
\begin{tikzcd}
H^1 (K,\Ind ^{K_f }_K \F _2 )\times H^1 (K,\End (\Ind ^{K_f }_K \F _2 ))  \arrow[r] \arrow[d] & H^2 (K,\Ind ^{K_f }_K \F _2 ) \arrow[d] \\
K_f ^\times \otimes \F _2 \times\left( K_{f,2}^\times \otimes \F _2 \times K_f ^\times \otimes \F _2 \right)  \arrow[r] & H^2 (K_f ,\F _2 ) \\
\end{tikzcd}
\]
commutes. From Proposition \ref{prop:cup_end}, we deduce that, for $c_1 \in H^1 (K,\Ind ^{K_f }_K \F _2 )$ and $c_2 \in H^1 (K,\Ind ^{K_f }_{K}(\Res ^{K_f }_K \Ind ^{K_f }_K \F _2 ))$ corresponding to $\overline{c}_1 $ and $\overline{c}_2 $ in $H^1 (K_f ,\F _2 )$ and $H^1 (K_f ,\Res ^{K_f }_K \Ind ^{K_f }_K \F _2 )$ respectively, we have
\[
c_1 \cup c_2 =\cores _{K_f |K}(\overline{c}_1 \cup \overline{c}_2 ).
\]
Now suppose that $\overline{c}_2 $ corresponds to $(c_3 ,c_4 )\in H^1 (K_{f,2},\F _2 )\times H^1 (K_f ,\F _2 )$ under the isomorphism 
\[
H^1 (K,\End (\Ind ^{K_f }_K \F _2 ))\simeq K_{f,2}^\times \otimes \F _2 \times K_f ^\times \otimes \F _2 . 
\]
Then, by Lemma \ref{lemma:compatibility_with_cup}, under the isomorphism $H^2 (K_f ,\Ind ^{K_f }\F _2 )\simeq H^2 (K_{f},\Ind ^{K_{f,2}}_{K_f }\F _2 ) \oplus H^2 (K_f ,\F _2 )$ we have
\[
\overline{c}_1 \cup \overline{c}_2 =(c_3 \cup i_1 (c_2 ),c_4 \cup c_2 ).
\]
Applying Proposition \ref{prop:cup_end}, we deduce part (1) of Proposition \ref{prop:another_explicit_description}.

\subsection{Describing the class of $J[4]$}
We now prove part (2) of Proposition \ref{prop:another_explicit_description}, namely that the class of $[J[4]]-[(\Ind ^{K_f }_K \mu _4 )/\mu _4 ]$ in $H^1 (K,\End (J[2]))$ is equal to $(\alpha -\beta )f'(\beta )$ with respect to the isomorphism from part (1) of Proposition \ref{prop:another_explicit_description}. First, note that the class of 
\[
[(\Ind ^{K_f }_K \mu _4 )/\mu _4 ]-[(\Ind ^{K_f }_K \Z /4\Z )/\Z /4\Z ]
\]
is equal to $-1$ in $K_{f,2}$. Hence it will be enough to show that the class of $[J[4]]-[(\Ind ^{K_f }_K \Z /4\Z )/\Z /4\Z ]$ is $(\beta -\alpha )f'(\beta )$. 

To do this, we recall some results from \cite[\S 3]{BKdescent1}. The $\Z /2\Z $ cover corresponding to a root $\alpha $ of $f$ is given on function fields by $\overline{K}(X)(u_{\alpha })$, where $c_{\alpha }u_{\alpha }^2 =x-\alpha $, and $c_{\alpha }=x(b)-\alpha $, where $b\in X(K)$ is a non-Weierstrass point.

Let $z_{\beta }$ be a square root of
\[
\prod _{\gamma \in \Roots (f)-\{ \beta \} }\left(1+\frac{c_{\beta }u_{\beta }-c_{\gamma }u_{\gamma }}{\gamma -\beta }\right).
\]
By \cite[Lemma 11]{BKdescent1} and \cite[Lemma 12]{BKdescent1}, the extension of $\overline{\Q }(X)$ corresponding to $J[4]$ is equal to $L_0 =\overline{\Q }(X)(u_{\alpha },z_{\alpha }:\alpha \in \Roots (f))$. In particular we have an isomorphism of Galois modules
\[
J[4]\simeq \Gal (L_0 |\overline{\Q }(X))
\]
where the action of $\Gal (K )$ on the latter is via the outer action of $\Gal (K(X))$ via conjugation. We recall \cite[Lemma 10]{BKdescent1} that the \'etale algebra $\bigotimes _{\overline{\Q }(X),\alpha \in \Roots (f)}\overline{\Q }(X)[u_{\alpha }]$ is not a field, and we are identifying $L_0 $ with the field in which 
\begin{equation}\label{eqn:recall}
\prod _{\alpha \in \Roots (f)}u_\alpha =y/y(b).
\end{equation}

Let $L_1 =K(X)(u_{\gamma }:\gamma \neq \alpha ,\beta )$, and $L_2 =L_1 (u_{\alpha },u_{\beta },z_{\beta })$. Then we have an outer action of $\Gal (K_{f,2}(X))$ on $\Gal (L_2 |L_1 )$.
\begin{lemma}
Suppose the Galois group of $K$ acts $2$-transitively on the roots of $f$ and $X$ has a $K$-rational non-Weierstrass point. Then under the isomorphism 
\[
H^1 (K,\Hom (J[2],\Ind ^{K_f }_K \F _2 ))\simeq H^1 (K_{f,2},\F _2 ),
\]
the class of the $\End (J[2])$-torsor of isomorphisms
\[
J[4]\simeq (\Ind ^{K_f }_K \Z /4\Z )/\Z /4\Z 
\]
is sent to the $\Z /2\Z $-torsor of isomorphisms
\[
\Gal (L_2 |L_1 )\simeq \Z /4\Z .
\] 
\end{lemma}
\begin{proof}
This follows from twice applying the explicit description of Shapiro's lemma from Lemma \ref{lemma:concrete_shapiro} to the isomorphism \eqref{eqn:swap} obtained from Lemma \ref{lemma:another_shapiro}.
\end{proof}

The main calculation in the proof of part (2) of Proposition \ref{prop:another_explicit_description} is the following. 
\begin{lemma}\label{lemma:bijection}
There is a $\Gal (K_{f,2})$-equivariant bijection
\[
\Iso (\Gal (L_2 |L_1 ),\Z /4\Z )\simeq \Roots (x^2 -f'(\beta )(\beta -\alpha ))
\]
\end{lemma}
\begin{proof}
A generator of $\Gal (L_2 |L_1 )$ is given by a lift of the generator of $\Gal (L_1 (u_{\alpha })|L_1 )$ to $\Gal (L_2 |L_1 )$. The generator of $\Gal (L_1 (u_{\alpha })|L_1 )$ is the automorphism $u_{\alpha }\mapsto -u_{\alpha }$, and a lift $\sigma $ to $\Gal (L_2 |L_1 )$ is uniquely determined by where it sends $z_{\beta }$. Note that by \eqref{eqn:recall}, $\sigma $ sends $u_{\beta }$ to $-u_{\beta }$. Hence we see that 
\begin{align*}
& z_{\beta }^2 \sigma (z_{\beta })^2  \\
= & (1+\frac{c_{\beta }u_{\beta }-c_{\alpha }u_{\alpha }}{\alpha -\beta })(1-\frac{c_{\beta }u_{\beta }-c_{\alpha }u_{\alpha }}{\alpha -\beta })\cdot \prod _{\gamma \neq \alpha ,\beta }(1+\frac{c_{\beta }u_{\beta }-c_{\gamma }u_{\gamma }}{\gamma -\beta })(1+\frac{-c_{\beta }u_{\beta }-c_{\gamma }u_{\gamma }}{\gamma -\beta })
\end{align*}
We have
\[
(1+\frac{c_{\beta }u_{\beta }-c_{\gamma }u_{\gamma }}{\gamma -\beta })(1+\frac{-c_{\beta }u_{\beta }-c_{\gamma }u_{\gamma }}{\gamma -\beta })=\frac{c_{\gamma }}{\beta -\gamma }(u_{\gamma }+1)^2 ,
\]
and
\[
(1+\frac{c_{\beta }u_{\beta }-c_{\alpha }u_{\alpha }}{\alpha -\beta })(1-\frac{c_{\beta }u_{\beta }-c_{\alpha }u_{\alpha }}{\alpha -\beta })=-c_{\beta }c_{\alpha }\left(\frac{(u_{\beta }-1)(u_{\alpha }+1)}{(\alpha -\beta -c_{\beta }u_{\beta }-c_{\alpha }u_{\alpha })}\right)^2 .
\]
We deduce that
\[
 \sigma (z_{\beta }) 
 = \sqrt{\frac{f'(\beta )}{\beta -\alpha }}\frac{\left(\frac{(u_{\beta }-1)(u_{\alpha }+1)}{(\alpha -\beta -c_{\beta }u_{\beta }-c_{\alpha }u_{\alpha })}\right)y(b)\prod _{\gamma \neq \alpha ,\beta }(u_{\gamma }+1)}{z_{\beta }}
\]
for some choice of square root, giving the desired Galois-equivariant bijection.
\end{proof}

\section{The boundary map for $\wedge ^2 J[4]$}\label{sec:boundary2}
We now apply the above calculations to describe the boundary map
\[
H^1 (K,\wedge ^2 J[2])\to H^2 (K,\wedge^2 J[2])
\]
associated to the short exact sequence
\[
0\to \wedge ^2 J[2]\to \wedge ^2 J[4]\to \wedge ^2 J[2]\to 0.
\]
and 
\[
0\to \wedge ^2 J[2]\to \Ker (\Ind ^{K_f ^{(2)}}_K \mu _4 \to \Ind ^{K_f }_K \mu _4)\to \wedge^2 J[4]\to 0
\]

\subsection{A field-theoretic description of $H^1 (K,\End (\wedge ^2 J[2]))$}
We have an isomorphism
\[
H^1 (K,\End (\wedge ^2 \Ind ^{K_f }_K \F _2 )\simeq (K_f ^{(2)}\otimes K_f ^{(2)})^\times \otimes \F _2 .
\]
Using the direct sum decomposition of $\Ind ^{K_f }_K \F  _2 $, we can obtain a field-theoretic description of $H^1 (K,\End (\wedge ^2 J[2]))$. To do this, we need to introduce some notation for various \'etale algebras obtained from tensor powers of $K_f $. For any $n< \deg (f)$, we define $K_{f,n}$ to be the \'etale algebra obtained by adjoining $n$ distinct roots of $f$, i.e.
\[
K_{f,n}:=K[t_1 ,\ldots ,t_n ][\frac{1}{t_i -t_j }:i\neq j]/(f(t_1 ),\dots ,f(t_n )).
\]
We define $K_{f}^{(2,1,1)}$ to be the subfield of $K_{f,4}$ fixed by the involution swapping $\overline{t}_1 $ and $\overline{t}_2$ and define $K_{f,2}^{(2)}$ to be the subfield of $K_{f}^{(2,1,1)}$ fixed by the involution swapping $\overline{t}_3$ and $\overline{t}_4$.
We have isomorphisms
\begin{align*}
K_f ^{(2)}\otimes K_f ^{(2)} \simeq K_{f,2} ^{(2)} \times K_{f,3}\times K_f ^{(2)} \\
K_f ^{(2)}\otimes K_{f,2} \simeq K_f ^{(2,1,1)}\times K_{f,3}\times K_{f,3}\times K_{f,2} \\
K_f ^{(2)}\otimes K_f \simeq K_f ^{(2,1)}\times K_{f,2} \\
K_{f,2}\otimes K_f \simeq K_{f,3}\times K_{f,2}\times K_{f,2} \\
\end{align*}
Informally, thinking of $K_{f,n}$ as the algebra obtained by adjoining $n$-distinct roots of $f$ to $K$, these isomorphisms can be obtained from breaking $K_{f,n}\otimes K_{f,m}$ into a product of algebras of the form $K_{f,e}$, by considering the possible roots of $f$ that the two algebras have in common. On the level of Galois cohomology, we obtain the following.
\begin{lemma}\label{end_iso1}
We have an isomorphism
\[
H^1 (K,\End (\wedge ^2 J[2]))\simeq \Ker \left( (K_f ^{(2)}\otimes K_f ^{(2)})^\times \otimes \F _2 \to (K_f ^{(2)}\otimes K_f )^\times \otimes \F _2 /(K_f \otimes K_f )^\times \otimes \F _2 \right).
\]
\end{lemma}
\begin{proof}
This follows from an iterated application of Shapiro's lemma:
\begin{align*}
H^ (K,\End (\wedge ^2 \Ind ^{K_f }_K \F _2 )) & \simeq H^1 (K,\Ind ^{K_f ^{(2)}}_K \F _2 \otimes \Ind ^{K_f ^{(2)}}_K \F _2 ) \\
& \simeq  H^1 (K,\Ind ^{K_f ^{(2)}}_K (\Res ^{K_f ^{(2)}}_K \Ind ^{K_f ^{(2)}}_K \F _2 ) \\
& \simeq H^1 (K_f ^{(2)},\Res ^{K_f ^{(2)}}_K \Ind ^{K_f ^{(2)}}_K \F _2 ) \\
& \simeq H^1 (K_f ^{(2)},\Ind ^{K_{f,2}^{(2)}}_{K_f ^{(2)}}\F _2 \oplus \Ind ^{ K_{f,3}}_{K_f ^{(2)}}\F _2 \oplus \F _2 ) \\
& \simeq K_{f,2} ^{(2),\times }\otimes \F _2  \oplus K_{f,3}^\times \otimes \F _2 \oplus K_f ^{(2),\times } \otimes \F _2 
\end{align*}
\end{proof}

We define maps $\pi $ and $\iota $ by the following commutative diagrams 
\[
\begin{tikzcd}
K_f ^{(2)}\otimes K_f ^{(2)} \arrow[d] \arrow[r, "1\otimes i "] & K_f ^{(2)}\otimes K_{f,2} \arrow[d] \\
K_{f,2}^{(2)}\times K_{f,3}\times K_f ^{(2)} \arrow[r, "\iota "]           & K_f ^{(2,1,1)}\times K_{f,3}\times K_{f,3}\times K_{f,2}         
\end{tikzcd}
\]
and
\[
\begin{tikzcd}
K_f ^{(2)}\otimes K_{f,2} \arrow[d] \arrow[r, "1\otimes \Nm "] & K_f ^{(2)}\otimes K_f \arrow[d] \\
K_f ^{(2,1,1)}\times K_{f,3}\times K_{f,3}\times K_{f,2}  \arrow[r, "\pi "]           & K_f ^{(2,1)}\times K_{f,2},
\end{tikzcd}
\]
where the vertical maps are the isomorphisms above. Then $\iota $ is the map 
\[
(f_1 ,f_2 ,f_3 )\mapsto (f_1 ,f_2 ,f_2 ,f_3 )
\]
and $\pi $ is the map
\[
(\Nm (g_1 )\Nm (g_2 ),g_4 \Nm (g_3 )).
\]

\subsection{Field-theoretic description of the map $H^1 (K,\End (J[2]))\to H^1 (K,\End (\wedge ^2 J[2]))$}
The direct sum decompositions $\Ind ^{K_f }_K \F _2 \simeq J[2]\oplus \F _2 $ and $\wedge ^2 \Ind ^{K_f }_K \F _2 \simeq \wedge ^2 J[2]\oplus J[2]$ induce maps
\begin{align*}
& \End (\Ind ^{K_f }_K \F _2 )\to \End (J[2])\to \End (\Ind ^{K_f }_K \F _2 ) \\
& \End (\wedge ^2 \Ind ^{K_f }_K \F _2 )\to \End (\wedge ^2 J[2])\to \End (\wedge ^2 \Ind ^{K_f }_K \F _2 )
\end{align*}

\begin{lemma}\label{lemma:theta_explicit}
The map
\[
H^1 (K,\End (\Ind ^{K_f }_K \F _2 )) \to H^1 (K,\End (\wedge ^2 \Ind ^{K_f }\F _2 ))
\]
sends $(g(\alpha ,\beta ),h(\alpha ))$ to $(1,g(\alpha ,\gamma ),h(\alpha )h(\beta ))$.
\end{lemma}
\begin{proof}
The map 
\begin{equation}\label{eqn:the_map}
\End (\Ind ^{K_f }_K \F _2 ) \to \End (\wedge ^2 \Ind ^{K_f }\F _2 )
\end{equation}
sends $[\alpha _i ] ^* \otimes [\alpha _j ]$ to $\sum _{k\neq i,j} [\{ \alpha _i ,\alpha _k \}]^* \otimes [\{ \alpha _k ,\alpha _j \} ]$. Under the identifications
\[
\End (\Ind ^{K_f }_K \F _2 )\simeq \Ind ^{K_{f,2}}_K \F _2 \oplus \Ind ^{K_f }_K \F _2
\]
and 
\[
\End (\wedge ^2 \Ind ^{K_f }_K \F _2 )\simeq \Ind ^{K_{f,2}^{(2)}}_K \F _2 \oplus \Ind ^{K_{f,3} }_K \F _2\oplus \Ind ^{K_{f}^{(2)} }_K \F _2 ,
\]
\eqref{eqn:the_map} corresponds to the direct sum $f_1 \oplus f_2 $, where $f_1 $ is the map induced by the inclusion of $K_{f,2}$ into $K_{f,3}$ via sending $\overline{s}$ and $\overline{t}$ to the first and third parameters in $K_{f,3}$, and $f_2 $ is the composite of the inclusion of $K_f$ into $K_{f,2}$ with the norm from $K_{f,2}$ to $K_f ^{(2)}$.
\end{proof}
Given a vector space $V$, we have a map
\[
\Theta _V :\End (V)\to \End (\wedge ^2 V)
\]
given by sending $\theta \in \End (V)$ to
\[
v\wedge w\mapsto \theta (v)\wedge w-v\wedge \theta (w).
\]
From the following lemma we deduce that the extension class of $\wedge ^2 J[4]$ is essentially given by $\Theta _{J[2]}$ of the class of $J[4]$.
\begin{lemma}\label{lemma:boundary1}
Let $M$ be a Galois module, and suppose that $E_1 $ and $E_2 $ are self-extensions of $M$, such that $E_1$ is isomorphic to the twist of $E_2 $ by $c\in Z^1 (K,\End (M))$. Then $\wedge ^2 E_1 $ is isomorphic to the twist of $\wedge ^2 E_2 $ by $\Theta (c)\in Z^1 (K,\End (\wedge ^2 M))$. In particular, if $\delta _i $ denotes the boundary map for $\wedge ^2 (E_i )$, then
\[
\delta _1 =\Theta (c) \cup (.)+\delta _2 .
\]
\end{lemma}
\begin{proof}
The first claim can be seen on the level of cocycles. The second claim is a consequence of Lemma \ref{lemma:boundary_cup}.
\end{proof}
To give an explicit formula for the cup product, it is enough to give an explicit formula for the class in $H^1 (K,\End (\wedge ^2 \Ind ^{K_f }_K \F _2 ))$. There is a minor subtlety here, since the diagram
\[
\begin{tikzcd}
\End (J[2]) \arrow[r, "\Theta _{J[2]}" ] \arrow[d] &  \End (\wedge ^2 J[2]) \arrow[d]          \\
\End (\Ind ^{K_f }_K \F _2 ) \arrow[r, "\Theta _{\Ind ^{K_f }_K \F _2 }" ]           & \End (\wedge ^2 \Ind ^{K_f }\F _2 )
\end{tikzcd}
\]
does not commute. However this issue disappears if we replace $\End (\wedge ^2 \Ind ^{K_f }\F _2 )$ with $\Hom (\wedge ^2 J[2],\wedge ^2 \Ind ^{K_f }\F _2 )$ via projection.

\begin{lemma}\label{lemma:commutes}
The diagram 
\[
\begin{tikzcd}
H^1 (K,\End (J[2])\times H^1 (K,\wedge ^2 J[2]) \arrow[r] \arrow[d, "\Theta _{J[2]}\times 1" ] & H^1 (K,\End (\Ind ^{K_f }_K \F _2 )\times H^1 (K,\wedge ^2 \Ind ^{K_f }_K \F _2 ]) \arrow[d, "\Theta _{\Ind ^{K_f }_K \F _2 } \times 1 " ] \\
H^1 (K,\End (\wedge ^2 J[2])\times H^1 (K,\wedge ^2 J[2])  \arrow[r] \arrow[d] & H^1 (K,\End (\wedge ^2 \Ind ^{K_f }_K \F _2 )\times H^1 (K,\wedge ^2 \Ind ^{K_f }_K \F _2 ) \arrow[d]  \\ 
H^2 (K,\wedge ^2 J[2])     \arrow[r] &  H^2 (K,\wedge ^2 \Ind ^{K_f }_K \F _2 )    \\
\end{tikzcd}
\]
commutes.
\end{lemma}
\begin{proof}
The map
\[
\wedge ^2 J[2]\otimes \End (\wedge ^2 \Ind ^{K_f }_K \F _2 )\to \Ind ^{K_f }_K \F _2 
\]
factors through $\wedge ^2 J[2]\otimes \Hom (\wedge ^2 J[2],\wedge ^2 \Ind ^{K_f }_K \F _2 )$, so it is enough to check commutativity of
\[
\begin{tikzcd}
\End (J[2]) \arrow[r] \arrow[d] &  \End (\wedge ^2 J[2]) \arrow[d]          \\
\Hom (J[2],\Ind ^{K_f }_K \F _2 ) \arrow[r]           & \Hom (\wedge ^2 J[2],\wedge ^2 \Ind ^{K_f }\F _2 )
\end{tikzcd}
\]
which follows from the definition of $\Theta $.
\end{proof}

\subsection{The boundary map for $\wedge ^2 J[4]$ in terms of cup products}
We now wish to describe the boundary map for $\wedge ^2 J[4]$ in terms of the cup product with $\Theta $ applied to the class of $J[4]$.
Let $G<\Sym (\Roots (f))$ denote the subgroup of the symmetric group on the set of roots of $f$ stabilising the set $\{ \alpha ,\beta \}$, and let $H<G$ denote the stabiliser of the ordered pair $(\alpha ,\beta )$. Let $\chi $ denote the $G$-module which is a free rank one $\Z /4\Z $-module with the unique nontrivial action of $G/H$. Abusing notation, we shall denote by $\chi $ the associated sheaf on $\Spec (K_f ^{(2)})_{\et }$.
\begin{lemma}\label{lemma:chi}
We have isomorphism of Galois modules
\begin{align*}
& \wedge ^2 \Ind ^{K_f }_K \Z /4\Z \simeq (\Ind ^{K_{f,2}}_K \Z /4\Z )/\Ind ^{K_f ^{(2)}}_K \Z /4\Z \\
& \simeq \Ind ^{K_f ^{(2)}}_K \chi .
\end{align*}
\end{lemma}
\begin{proof}
We prove each of the isomorphisms in turn. For the first isomorphism, we take the map
\[
\Ind ^{K_{f,2}}_K \Z /4\Z \to \wedge ^2 \Ind ^{K_f }_K \Z /4\Z
\]
sending $[ \{ \epsilon _1 ,\epsilon _2 \} ]$ to $[\{ \epsilon _1 \}]\wedge [\{\epsilon _2 \}]$. The second isomorphism is the induction from $K_f ^{(2)}$ to $K$ of the isomorphism of $\Gal (K_f ^{(2)})$-modules
\[
(\Ind ^{K_{f,2}}_{K_f ^{(2)}}\Z /4\Z )/\Z /4\Z \simeq \chi
\]
which is purely a statement about the group algebra of $\Z /4\Z [\Z /2\Z ]$.
\end{proof}

We deduce the following lemma.
\begin{lemma}\label{lemma:boundary2}
The boundary map $H^1 (K,\wedge ^2 J[2])\to H^2 (K,\wedge ^2 J[2])$ associated to $\wedge ^2 \Ker (\Ind ^{K_f }_K \Z /4\Z \to \Z /4\Z )$ is given by 
\[
z\mapsto \langle -(\alpha ^2 +\beta ^2 -2\alpha \beta ),z \rangle _{K_f ^{(2)}}.
\]
\end{lemma}
\begin{proof}
By Lemma \ref{lemma:chi}, the class of $[\wedge ^2 \Ind ^{K_f }_K \Z /4\Z ]-[\Ind ^{K_f ^{(2)}}\Z /4\Z ]$ in $H^1 (K,\End (\wedge ^2 \Ind ^{K_f }_K \Z /2\Z ))$ is the image of the class of $(\alpha -\beta )^2 $ in $K_f ^{(2),\times }\otimes \F _2 $ under the diagonal map
\[
H^1 (K_f ^{(2)},\Z /2\Z )\to H^1 (K,\End (\Ind ^{K_f ^{(2)}}_K \F _2 )).
\]
Via the field-theoretic descriptions of the source and target this corresponds to the map
\[
K_f ^{(2),\times }\otimes \Z /2\Z \to K_{f,2}^{(2),\times }\otimes \F _2 \oplus K_{f,3}^\times \otimes \F _2 \oplus K_f ^{(2),\times }\otimes \F _2 
\]
given by $w\mapsto (1,1,w)$. Similarly, the class of $[\Ind ^{K_f ^{(2)}}_K \Z /4\Z ]-[\Ind ^{K_f ^{(2)}}_K \mu _4 ]$ corresponds to $(1,1,-1)$. Since the boundary map for $\Ind ^{K_f ^{(2)}}_K \mu _4 $ is zero, we deduce that the boundary map for $\wedge ^2 \Ind ^{K_f }_K \Z /4\Z $ is given as above. We have a direct sum decomposition
\[
\wedge ^2 \Ind ^{K_f }_K \Z /4\Z \simeq I\oplus \wedge ^2 I,
\]
where $I:=\Ker (\Ind ^{K_f }_K \Z /4\Z \to \Z /4\Z )$. Hence the boundary map for $\Ind ^{K_f }_K \Z /4\Z $ is also the direct sum of the boundary maps for $\wedge ^2 I$ and $I$.
\end{proof}
We have a surjection
\[
\Ind ^{K_f ^{(2)}}_K \Z /4 \Z \to \Ker (\Ind ^{K_f }_K \Z /4\Z \to \Z /4\Z )
\]
given by sending $\{ \epsilon _1 ,\epsilon _2 \}$ to $[\epsilon _1 ]+[\epsilon _2 ]+2\sum _{\epsilon _3 \in \Roots (f)}[\epsilon _3 ]$. This has a section given by 
\[
[\epsilon _1 ]-[\epsilon _2 ]\mapsto (-1)^{g+1}\sum _{\epsilon _3 \neq \epsilon _1 ,\epsilon _2 }\{\epsilon _1 ,\epsilon _3 \}-\{ \epsilon _2 ,\epsilon _3 \}
\]
We deduce that 
\[
A:=\Ker (\Ind ^{K_f ^{(2)}}_K \mu _4 \to \Ker (\Ind ^{K_f }_K \mu _4 \to \mu _4 ))
\]
is a direct summand of $\Ind ^{K_f ^{(2)}}_K \mu _4 $ and hence that $A$ is a self-extension of $\wedge ^2 J[2]$ for which the boundary map $H^1 (K,\wedge ^2 J[2])\to H^2 (K,\wedge ^2 J[2])$ is zero. Let $\alpha ,\beta ,\gamma $ denote the images of $t_1 ,t_2 ,t_3 $ in $K_{f,3}$. Let $i_1 $ and $i_2 $ denote the inclusions $K_f ^{(2)}\to K_{f,3}$ sending $(t_1 ,t_2 )$ to $(t_1 ,t_2 )$ and $(t_2 ,t_3 )$ respectively. We will also denote by $i_1 $ and $i_2$ the composite with $K_{f,3}\to K_f ^{(2)}\otimes K_f ^{(2)}$. We let $i_3 $ denote the inclusion $K_f ^{(2)}\to K_f ^{(2)}\otimes K_f ^{(2)}$ defined above.
\begin{proposition}\label{prop:wedge_boundary_general}
The boundary map
\[
H^1 (K,\wedge ^2 J[2])\to H^2 (K,\wedge ^2 J[2])
\]
associated to $\wedge ^2 J[4]$ is given by
\[
z\mapsto \cores _{K_{f,3}|K_f ^{(2)}}^{(2)}\langle (\gamma -\alpha )f'(\gamma ),i_1 (z) \rangle _{K_{f,3}} +\langle -(\alpha ^2 +\beta ^2 -2\alpha \beta ),z \rangle _{K_f ^{(2)}}.
\]
\end{proposition}
\begin{proof}
By the direct sum decomposition above, the boundary map sends $z$ to
\[
\cores ^{(2)}_{K_f ^{(2)}\otimes K_f ^{(2)}|K_f ^{(2)}}\langle \kappa ,i_2 (z) \rangle _{K_f ^{(2)}\otimes K_f ^{(2)}},
\]
where $\kappa \in K_f ^{(2),\times }\otimes K_f ^{(2),\times }$ is the image of $[A]-[\wedge ^2 J[4]]$ under the map
\[
H^1 (K,\End (\wedge ^2 J[2]))\to (K_f ^{(2),\times }\otimes K_f ^{(2),\times })^{\times } \otimes \F _2 .
\]
Hence it is enough to compute $\kappa $. By Proposition \ref{prop:everything}, together with Lemmas \ref{lemma:theta_explicit}, \ref{lemma:commutes},\ref{lemma:boundary1}, \ref{lemma:chi} and \ref{lemma:boundary2}, we have 
\[
\kappa =i_3 (-(\alpha ^2 +\beta ^2 -2\alpha \beta ))\cdot \Theta ((\beta -\alpha )f'(\beta )).
\]
The proposition follows.
\end{proof}
\subsection{The boundary map for local fields}
As for $J[4]$, the boundary map for $\wedge ^2 J[4]$ admits a simpler description when the field $K$ is $p$-adic. Given a decomposition of $K_{f,3}$ into fields $\prod L_i$, we let $\alpha _i ,\beta _i$ and $\gamma _i$ denote the images of $t_1 ,t_2 $ and $t_3 $ in $L_i $.

\begin{proposition}\label{prop:wedge_boundary}
Let $K$ be a finite extension of $\Q _p $. Let $K_f ^{(2)}=\prod _{i=1}^m K_i$, and let $K_{f,3}=\prod _{i=1}^n L_i$ be decompositions into products of fields. Define $\pi _1 ,\pi _2 :\{ 1,\ldots ,n\} \to \{1,\ldots ,m \}$ by the property that the two maps $K_f ^{(2)}\to K_{f,3}$ send $K_{\pi _1 (i)}$ and $K_{\pi _2 (i)}$ respectively to $L_i$. Then the boundary map
\[
H^1 (K,\wedge ^2 J[2])\to H^2 (K,\wedge ^2 J[2])
\]
sends $(z_i )\in H^1 (K,\wedge ^2 J[2])\subset \prod _i K_i ^\times \otimes \F _2 $ to $(w_i )\in H^2 (K,\wedge ^2 J[2])\subset \prod _i \Br (K_i )[2]$, where
\[
w_i =\langle -(\alpha _i ^2 +\beta _i ^2 -2\alpha _i \beta _i ),z_j \rangle _{K_j }+\sum _{\pi _2 (k)=i}\langle \Nm _{L_k|K_{\pi _2 (k)} }(\gamma _k -\alpha _k )f'(\gamma _k ),z_{\pi _1 (k)} \rangle _{K_{\pi _1 (k)}}.
\]
\end{proposition}
\begin{proof}
By Proposition \ref{prop:wedge_boundary_general}, the boundary map sends $(z_i )$ to $\cores _{K_{f,3}|K_f ^{(2)}}^{(2)}\langle (\gamma -\alpha )f'(\gamma ),i_1 (z) \rangle _{K_{f,3}} +\langle -(\alpha ^2 +\beta ^2 -2\alpha \beta ),z \rangle _{K_f ^{(2)}}$. By definition, this is equal to $(w_i )$, where $w_i \in \Br (K_i )[2]$ is given by
\[
\langle -(\alpha _i ^2 +\beta _i ^2 -2\alpha _i \beta _i ),z_i \rangle _{K_i }+\sum _{\pi _2 (k)=i} \cores ^{(2)}_{L_k |K_i }\langle (\gamma _k -\alpha _k )f'(\gamma _k ),\iota _k (z_{\pi _1 (k)})\rangle _{L_k }. 
\]
By Lemma \ref{lemma:LCFT_boundary}, we have
\begin{align*}
& \cores ^{(2)}_{L_k |K_i }\langle (\gamma _k -\alpha _k )f'(\gamma _k ),\iota _k (z_{\pi _1 (k)})\rangle _{L_k } \\
& =\cores ^{(2)}_{L_k |K}
\langle (\gamma _k -\alpha _k )f'(\gamma _k ),\iota _k (z_{\pi _1 (k)})\rangle _{L_k }\\
& = \cores ^{(2)}_{L_k |K_{\pi _1 (k)}}\langle (\gamma _k -\alpha _k )f'(\gamma _k ),\iota _k (z_{\pi _1 (k)})\rangle _{L_k },
\end{align*}
and the last expression is equal to 
\[
\langle \Nm _{L_k|K_{\pi _2 (k)} }(\gamma _k -\alpha _k )f'(\gamma _k ),z_{\pi _1 (k)} \rangle _{K_{\pi _1 (k)}}.
\]
\end{proof}
For a prime $p$ we denote by $\theta _p $ the map 
\[
\Ker (\Q _{p,f} ^{(2),\times }\otimes \to \Q _{p,f} ^\times )\to \Ker (\Br (\Q _{p,f} ^{(2)})[2]\to \Br (\Q _{p,f} )[2])
\] 
from Proposition \ref{prop:wedge_boundary}.

\section{Lifting obstructions at $\infty $}\label{sec:infty}
We now apply the description of these boundary maps to obtain lifting obsructions at $\infty $, i.e. to compute the boundary map
\[
H^1 (\R ,\wedge ^2 J[2])\to H^2 (\R ,\wedge ^2 T_2 J).
\]
\subsection{Local aspects at infinity: dimensions}\label{subsec:infinity}
Suppose $M$ is a torsion-free abelian group with an indecomposable action of $\mathbb{Z}/2\mathbb{Z}=\langle c \rangle$. Then there are exactly three options for $M$. Either $M=\mathbb{Z}$ with the trivial action, $M=\mathbb{Z}(1)$ (i.e. rank 1, and $c$ acts as $-1$) or $M=M_c:=\mathbb{Z}[\mathbb{Z}/2\mathbb{Z}]$. It is straightforward to verify the following computations.
\begin{lemma}\label{lemma:R}
\begin{enumerate}
\item $M_c \otimes \mathbb{Z}(1)\simeq M_c $,
\item $M_c \otimes M_c \simeq M_c ^{\oplus 2} $,
\item $\wedge ^2 M_c \simeq \mathbb{Z}(1) $,
\item $H^1 (\mathbb{Z}/2\mathbb{Z},\mathbb{Z}_2 )=0.$
\item $H^1 (\mathbb{Z}/2\mathbb{Z},\mathbb{Z}_2 (1) )=\mathbb{Z}/2\mathbb{Z}.$
\item $H^1 (\mathbb{Z}/2\mathbb{Z},\mathbb{Z}_2 \otimes M_c )=0.$
\item $H^1 (\mathbb{Z}/2\mathbb{Z},\F _2 \otimes M_c )=0$.
\end{enumerate}
\end{lemma}

\begin{lemma}
Let $f$ be an odd degree separable polynomial in $\mathbb{R}[x]$ with exactly $r_1 (f)$ real roots. Let $J$ denote the Jacobian of 
\[
T_2 (\Jac (X))\simeq (\mathbb{Z}_2 \oplus \mathbb{Z}_2 (1))^{(r_1 (f)-1)/2}\oplus M_c \otimes \mathbb{Z}_2 ^{g-(r_1 (f)-1)/2}.
\]
\end{lemma}
\begin{proof}
It is enough to prove the corresponding claim for $H_1 (X(\mathbb{C}),\mathbb{Z})$. By the classification of indecomposable $\mathbb{Z}/2\mathbb{Z}$-modules, $H_1 (\mathbb{X}(\mathbb{C}),\mathbb{Z})$ is isomorphic to $\mathbb{Z}^a \oplus \mathbb{Z}(1)^b \oplus M_c ^d$,where $a+b+2c=2g$. Furthermore we must have $a+d=g$, and by the Lemma we have
\[
a+b+d=\dim H^0 (\mathbb{R},J[2]).
\]
On ther other hand we also have 
\[
\dim H^0 (\mathbb{R},J[2])=r_1 (f)+r_2 (f)-1.
\]
The lemma follows from rearranging these formulas.
\end{proof}
We arrive at the following Lemma, which gives a liftability obstruction when $f$ has more than one real root.
\begin{lemma}\label{lemma:wedge2realdimension}
For $f$ and $J$ as above,
\[
\dim _{\F _2 } H^1 (\Gal (\mathbb{C}|\mathbb{R}),\wedge ^2 T_2 J)=\binom{(r_1 (f)-1)/2}{2}+g,
\]
and 
\[
\dim _{\F _2 } H^1 (\Gal (\mathbb{C}|\mathbb{R}),\wedge ^2 J[2])=\binom{2(r_1 (f)-1)/2}{2}+g.
\]
\end{lemma}
\begin{proof}
By Lemma \ref{lemma:R}, we deduce 
\[
\wedge ^2 T_2 J\simeq \mathbb{Z}(1)^{\binom{(r_1 (f)-1)/2}{2}+g}\oplus \mathbb{Z}^{\binom{(r_1 (f)-1)/2}{2}}\oplus M_c ^{(r_1 (f)-1+g)(5r_1 (f)+g-1)/2},
\]
and conclude from Lemma \ref{lemma:R}.
\end{proof}

\subsection{Local aspects at infinity: boundary maps}\label{subsec:infinity2}
In this subsection we use the description of boundary maps for $\wedge ^2 J[4]$ given in section \ref{sec:boundary2} when $K=\mathbb{R}$.

We note that if $X$ is a genus $g$ hyperelliptic curve with a rational Weierstrass point over $\R$, defined by an polynomial $f(x)\in \R [x]$ of degree $2g+1$, we can describe $H^1 (\R ,J[2])$ and the image of $H^1 (\R ,T_2 J)$ in $H^1 (\R ,J[2])$ explicitly in terms of the polynomial $f$. Namely, suppose $f$ has $2d+1$ real roots. Write
\[
f=c\cdot \prod _{i=1}^{2d+1}(x-\alpha _i )\cdot \prod _{i=1}^{g-d}q_i (x)
\]
where $c\in \R ^\times $, $q_i$ are monic irreducible quadratics polynomials, and $\alpha _1 <\alpha _2 <\ldots <\alpha _n $. Then $H^1 (\R ,J[2])\simeq \Ker (\Nm :\{\pm 1\}^{2d+1}\to \{ \pm 1\} )$, where the norm map is just multiplication. Given non-Weierstrass points $b,z$, the class of $z-b$ in $H^1 (\R ,J[2])$ is given by the tuple $(\sgn \left( \frac{x(z)-\alpha _i }{x(b)-\alpha _i } \right) )_{i=1}^{2d+1}$. If $c>0$, then $x=\lambda $ defines a real point of $X$ if and only if $\lambda >\alpha _i$ for an even number of $\alpha _i$. It follows that the subgroup of $H^1 (\R ,J[2])$ generated by $X(\R )-X(\R )$ has dimension $d$ (and hence equals the image of $H^1 (\R ,T_2 J)$) and is equal to the subgroup of $\F _2 ^{2d+1}$ generated by
\[
e_2 +e_3 ,e_4 +e_5 ,\ldots ,e_{2d}+e_{2d+1}.
\]
Similarly if $c<0$ one deduces that the subgroup generated by $X(\R )-X(\R )$ is equal to the image of $H^1 (\R ,T_2 J)$, and is equal to the subgroup of $\F _2 ^{2d+1}$ generated by
\[
e_1 +e_2 ,e_3 +e_4 ,\ldots ,e_{2d-1}+e_{2d}.
\] 

Let's say $c>0$. Fix $\alpha $. Then $-f'(\alpha )<0$ iff the number of roots $>\alpha $ is even. So the condition on a tuple $(\epsilon _i )$ is that, for $i$ odd,
\[
\epsilon _i \cdot \prod _{j<i}\epsilon _j =1,
\]
and for $i$ even
\[
\prod _{j<i}\epsilon _j =1.
\]
So the condition is for all $i$,
\begin{equation}\label{eqn:epsilon1}
\epsilon _{2i}\epsilon _{2i+1}=1.
\end{equation}
Similarly when $c<0$ the condition is
\begin{equation}\label{eqn:epsilon2}
\epsilon _{2i-1}\epsilon _{2i}=1.
\end{equation}
This gives another way of seeing that the boundary obstruction exactly recovers the image of $H^1 (\R ,T_2 J)$.

If we now look at the wedge square, we see that a basis of the $\Z _2 $-summand is given by $(e_{2i-1}+e_{2i} )\wedge (e_{2j-1}+e_{2j})$ and $(e_{2i}+e_{2i+1} )\wedge (e_{2j}+e_{2j+1})$ for $i< j$.

Write a factorisation of $f(x)$ over $\R[x]$ into irreducibles as
\[
f(x)=c\cdot \prod _{i=1}^{2d+1} (x-\alpha _i )\cdot \prod _{j=1}^{g-d}(x^2 +a_i x+b_i )
\]
Then, as explained above, we have an embedding
\[
\iota _{ij}:\Q _f ^{(2)}\hookrightarrow \R
\]
for each pair of real roots $\alpha _i ,\alpha _j $, given by sending $g(\alpha ,\beta )$ to $g(\alpha _i ,\beta _i )$ for each symmetric polynomial in $\alpha $ and $\beta $. For each $1\leq i\leq g-d$, we also have an embedding $\rho _i$ sending $\alpha \beta $ to $b_i$ and $\alpha +\beta $ to $-a_i$. For notational convenience for all $1\leq i\leq 2d+1$ we will also define $\iota _{ii}$ to be the map
\[
\Q _f ^{(2)}\to \{ 1\}.
\]
\begin{lemma}\label{lemma:thetaR}
Let $z$ be an element of $\Ker (\Q _f ^{(2),\times }\otimes \F _2 \to \Q _f ^{\times }\otimes \F _2 )$ $\simeq H^1 (\Q ,\wedge ^2 J[2])$ whose image in $H^1 (\R ,\wedge ^2 J[2])$ lifts to $H^1 (\R ,\wedge ^2 T_2 J)$. Then for all $i<j\leq d$, 
\[
\iota _{2i-1,2j-1}(z)\iota _{2i-1,2j}(z)\iota _{2i,2j-1}(z)\iota _{2i,2j}(z)>0.
\]
and 
\[
\iota _{2i,2j}(z)\iota _{2i+1,2j}(z)\iota _{2i,2j+1}(z)\iota _{2i+1,2j+1}(z)>0.
\]
\end{lemma}
\begin{proof}
This follows from \eqref{eqn:epsilon1} and \eqref{eqn:epsilon2}.
\end{proof}
We denote the map 
\begin{equation}\label{eqn:thetaR}
((\rho _i )_i ,\iota _{2i,2j-1}\cdot \iota _{2i+1,2j-1}\cdot \iota _{2i,2j}\cdot \iota _{2i+1,2j}):K_f ^{(2),\times }\otimes \F _2 \to (\R ^\times /\R _{>0})^{\frac{d(d-1)}{2}}
\end{equation}
by $\theta _{\R }$.

\section{Finiteness criteria}
The results recalled in section \ref{sec:recall} and proved in sections \ref{sec:boundary2} and \ref{sec:infty} combine to give the following criterion for finiteness of $X(\Q _2 )_2 $.
\begin{lemma}\label{lemma:finiteness_conditions}
Let $X$ be a genus $g$ hyperelliptic curve with a rational Weierstrass point, given by a polynomial $f(x)\in \Q [x]$ whose Galois group is a $2$-transitive subgroup of $S_{2g+1}$. Let $\theta _{\mathbb{R}}$ be the map from \eqref{eqn:thetaR}. Let $\theta _2 $ be the map from Proposition \ref{prop:wedge_boundary} (with $p=2$). Suppose
\begin{enumerate}
\item $\# \Cl (\mathcal{O}_{\Q _f } )[2]=\# \Cl (\mathcal{O}_{\Q _f ^{(2)}})[2].$
\item $X$ has semistable reduction at all primes away from $2$.
\item The rank of $\Ker (\theta _\R )\cap \Ker (\theta _2 )$ is less than $\frac{3g^2+g}{2}-\rk J(\Q )$.
\end{enumerate}
Then $X(\Q _2 )_2$ is finite.
\end{lemma}
\begin{proof}
Let 
\[
\pi :H^1 (\R ,\wedge ^2 T_2 J)\oplus H^1 (\Q _2 ,\wedge ^2 T_2 J)\to H^1 (\R ,\wedge ^2 J[2])\oplus H^1 (\Q _2 ,\wedge ^2 J[2])
\]
and 
\[
\loc :H^1 _{f,\{2\}}(\Q ,\wedge ^2 J[2])\to H^1 (\R ,\wedge ^2 J[2])\oplus H^1 (\Q _2 ,\wedge ^2 J[2])
\]
be the obvious maps. By Proposition \ref{prop:storysofar}, conditions (1) and (2) imply that the dimension of $H^1 _\{ 2\} (G_{\Q },\wedge ^2 V_2 J)$ is bounded by the dimension of $\loc ^{-1}(\mathrm{Image}(\pi ))$. By Proposition \ref{prop:wedge_boundary} and Lemma \ref{lemma:thetaR}, we have the containment
\[
\loc ^{-1}(\mathrm{Image}(\pi )) \subset \Ker (\theta _\R )\cap \Ker (\theta )_2 .
\] 
Note that we have not assumed that $X$ has a rational non-Weierstrass point over $\Q $, but if $\Gal (\Q )$ acts $2$-transitively on $f$ then we can always find a quadratic extension $K_0 |\Q $ which splits completely at $2$ and for which $X(K_0 )$ contains a non-Weierstrass point and $\Gal (K_0 )$ acts $2$-transitively on the roots of $f$.

Since $H^1 _{f,\{2\}} (\Q ,\wedge ^2 V_2 J) $ contains the one-dimensional image of $H^1 _{\{ 2\}}(\Q ,\Q _2 (1))$ of dimension one, the dimension of the complement $H^1 _{f,\{ 2\}}(G_{\Q },\wedge ^2 V_2 J/\Q _2 (1))$ is less than $\frac{(3g-2)(g+1)}{2}-\rk J(\Q )$. Hence the lemma follows from Lemma \ref{lemma:BK1_2}.
\end{proof}
Note that in the case when the number of primes above $2$ in $\Q _f ^{(2)}$ is equal to the number of primes above $2$ in $\Q _f $, condition $3$ is empty.

Although this is independent of the Diophantine results we consider in this paper, we note that $2$-descent methods are particularly effective at verifying the dimension conjectures of Bloch and Kato when one considers suitably large Tate twists of a Galois representations. More precisely, if $W$ is a crystalline representation of negative weight with $F^0 D_{\dR}(W)=0$, then the crystalline condition on Galois cohomology is empty, i.e. $H^1 (\Q _2 ,W)=H^1 _f (\Q ,W)$. This means that we can hope to have a field-theoretic 
\begin{proposition}\label{prop:tail}
Suppose that $f$ is degree $2g+1$ polynomial over $\Q $ which is irreducible over $\Q _2 $, that $K_f $ has odd class number, and the discriminant of $f$ is squarefree up to a power of $2$. Then for all $n>0$,
\[
\dim H^1 _f (G_{\Q },V_2 J(n))\geq g.
\]
\end{proposition}
\begin{proof}
The assumption on $n$ implies that $H^1 _f (G_{\Q _2 },V_2 J(n))=H^1 (G_{\Q _2 },V_2 J(n))$. The assumptions on the discriminant imply that inertia acts unipotently on $V_2 J(n)$ for all primes away from $2$. Putting these things together, we deduce that the submodule of classes in $H^1 (G_{\Q ,2},J[2])$ which lift to $H^1 (\R ,T_2 J(n))$ has dimension $g$, from which the proposition follows.
\end{proof}
The statistics for number fields $K_f$ of the form above were recently considered by Ho, Shankar and Varma \cite{HSV}. In particular they show that, subject to certain `tail estimates' a positive proportion (in an appropriate sense) of $f$ should satisfy these conditions. This suggests the possibility that one may apply the methods of arithmetic statistics to prove (parts of) the Bloch--Kato conjectures for a positive proportion of hyperelliptic curves.

\subsection{The relation with elliptic curve Chabauty}\label{sec:ecc}
We now explain the relation with elliptic curve Chabauty, as discussed in the introduction. Let $X$ be a genus $2$ curve over a field $K$ of characteristic different from $2$ with a $K$-rational Weierstrass point and let $f(x)\in K[x]$ be a degree $5$ polynomial giving a model of $X$. If $Y\to X$ is an \'etale $\Z /2\Z $-cover of $X$, then the Prym variety of $Y\to X$ is an elliptic curve. This elliptic curve may be realised as the Jacobian of a genus one quotient of $Y$. To explain this, we may work more generally / universally over $K_f$, and consider the curve 
\[
Y_{\alpha }:v_{\alpha }^2 =cc_{\alpha }\prod _{\beta \neq \alpha }(u_{\alpha }^2 -(\beta -\alpha )/c_{\alpha })
\]
over $K_f$. Then
\[
v_{\alpha }^2 =cc_{\alpha }\prod _{\beta \neq \alpha }(t_{\alpha }-(\beta -\alpha )/c_{\alpha })
\]
defines a genus one curve which is a quotient of $Y_{\alpha }$ given by quotienting out by the involution $u_{\alpha }\mapsto -u_{\alpha }$.

Note that the $2$-Selmer group of $\Jac (Y_{\alpha })$ is a subspace of 
\[
H^1 (K_f ,\Jac (Y _{\alpha })[2])\simeq H^1 (K,\Ind ^{K_f }_K \Jac (Y_{\alpha })[2]).
\]
The $\Gal (K^{\sep }|K)$-module $\Ind ^{K_f }_K \Jac (Y_{\alpha })[2]$ is isomorphic to $\Ind ^{K_{f,2}}_K \F _2 $, hence calculations in elliptic curve Chabauty are essentially happening in $K_{f,2}^\times \otimes \F _2 $ rather than $K_f ^{(2),\times }\otimes \F _2 $. A more significant difference is the Selmer conditions: the conditions of locally lifting to points on twists of $\Jac (Y_{\alpha })$ are not the same as the condition of lifting to $H^1 _f (G_{K_v },\wedge ^2 T_2 J)$. In particular, it can happen that the $2$-descent methods from this paper prove finiteness of $X(\Q _2 )_2$, but that there is a twist of the cover $Y_2 \to X$ (coming from a rational point of $X$) for which the corresponding $2$-Selmer group has large rank.

On the other hand, a common obstruction to both methods is $2$-torsion in the class group of $K_{f,2}$. If one were to compare four-descents for $\wedge ^2 T_2 J$ and the elliptic curves above presumably in some cases here one would again see a difference between the Tate--Shafarevich classes which survive.

\section{Examples}
\subsection{Statistics of rank two curves satisfying the conditions of Lemma \ref{lemma:finiteness_conditions}}
We tested the applicability of these algorithms on a list of genus two curves recently produced by the LMFDB \cite{LMFDB1}.
Using some of the formulas described in this paper, we found that at least 3,323 of the 7,224 genus 2 curves in the LMFDB of rank 2 with at least one rational Weierstrass point satisfied the condition $X(\Q _2 )_2 $ is finite. In fact we did not use all the conditions. We first restricted to curves whose Jacobians had no rational $2$-torsion (equivalently those admitting an odd-degree model with an irreducible Weierstrass polynomial). This reduces to 6,342 curves. 
Of those, we searched for curves satisfying the following conditions (note that some of these conditions depend on the choice of $f$ defining $X$, which was chosen based on the model of the curve provided by the LMFDB).
\begin{enumerate}
\item $f$ is not irreducible over $\Q $.
\item $f$ is in $\Z [\frac{1}{2}][x]$.
\item Let $g\in \Q _f [x]$ be an irreducible factor of $f$ not equal to $x-\alpha $. Let $\Q _f (\beta )$ be the field obtained by adjoining a root of $g$ to $\Q _f $. Let $H$ be an irreducible factor of $\Nm _{\Q _f [x]|\Q [x]}(\alpha +\beta )$. Then $\deg (H)= 10$.
\item For all primes $p>2$ dividing the discriminant of $f$, $p$ does not divide the $x^5$ coordinate of $f$, and $f$ is separable modulo $p$.
\item 
$
\val _2 (\mathrm{Cl}(\Q _f ^{(2)})) =\val _2 (\mathrm{Cl}(\Q _f )). 
$
\item 
$
\rk (\theta _2 \oplus \theta _\R )<\# \Spec (\Q _{2,f}^{(2)})-\Spec (\Q _{2,f})+d(d-1).
$\end{enumerate}
Note that if $f$ satisfies all these conditions, then it satisfies the conditions of Lemma \ref{lemma:finiteness_conditions}, and hence $X(\Q _2 )_2 $ is finite.

Next we explain how these conditions were computed. All computations were carried out using magma \cite{magma}. The first five conditions can easily be checked using functions in the magma library. For condition 6, we use Proposition \ref{prop:wedge_boundary} to reduce the computation of $\theta _2$ to the problem of calculating Hilbert symbols over finite extensions of $\Q _2$, which can be done using magma's Hilbert Symbol function.

In the table below, we record the implementation of this algorithm on the 6,603 hyperelliptic curves in the LMFDB with rank 2 and exactly one rational Weierstrass point. The numbers indicate the number of curves where the algorithm failed at this step (but passed all previous steps).
\begin{center}
 \begin{tabular}{|c|c|}
   \hline $\#\mathcal{S}$ & 6603 \\
   \hline $f$ not irreducible & 259 \\
   Coefficient issues & 22   \\
    $\Q _f ^{(2)}\neq \Q (\alpha +\beta )$ &  $14$ \\
    Bad primes & $843$ \\
    $(\mathrm{Cl}(\Q _f ^{(2)})/\mathrm{\Cl}(\Q _f ))[2]\neq 0$ & $762$ \\ 
    $\rk (\theta _2 \oplus \theta _\R )<\# \Spec (\Q _{2,f}^{(2)})-\Spec (\Q _{2,f})+d(d-1)$ & $ 1380$ \\ 
    \hline Verified $\# X(\Q _2 )_2 <\infty $ & 3323 \\ \hline
  \end{tabular}
    \end{center}
The most significant obstacle is the last one. In fact one can further break this down: for most of the curves is question $d\leq 1$, so $\theta _{\R }=0$ and the issue is the rank of $\theta _2$.

\subsection{Future developments}
There are a number of possible ways to extend these results to prove finiteness of $X(\Q _2 )_2$ for more curves. Conditions (2) and (4) are sufficient but not necessary to guarantee semistable reduction away from $2$. Condition (3) is sufficient but not necessary for $2$-transitivity of the action of $\Gal (\Q )$ on the roots of $f$. In fact conditions (1) and (3) could be removed completely if the condition of $2$-transitivity of the action of $\Gal (\Q )$ on the roots of $f$ were removed from Proposition \ref{prop:wedge_boundary}.

The following more serious obstacles remain.
\begin{enumerate}
\item Local obstructions away from 2: the discriminant of $f(X)$ is divisible by $p^2 $ for some $p\neq 2$.
\item Local obstructions at 2: the number of primes above $2$ in $\Q _f ^{(2)}$ is greater than the number of primes above $2$ in $\Q _f $, and this obstruction is not dealt with by $\Ker (\theta _2 )$.
\item Class group obstructions: the order of $\Cl (\mathcal{O}_{\Q _f ^{(2)}})[2]$ is greater than the order of $\Cl (\mathcal{O}_{\Q _f })[2]$.
\end{enumerate}
For the class group obstructions, one approach might be the following. Section \ref{sec:boundary} gives a description of the boundary map
\[
H^1 (G_{K,S},\wedge ^2 J[2])\to H^2 (G_{K,S},\wedge ^2 J[2])
\]
in terms of cup products of Galois cohomology with values in $\mu _2 $. This gives an obstruction to elements in $H^1 (G_{K,S},\wedge ^2 J[2])$ coming from the class group lifting to $H^1 (G_{K,S},\wedge ^2 J[4])$, given in terms of cup products
\[
H^1 (G_{L,S},\mu _2 )\otimes H^1 (G_{L,S},\mu _2 )\to H^2 (G_{L,S},\mu _2 )
\]
for various number fields $L$. Since the classes come from the class group, these cup products will give classes in $H^2 (G_{L,S},\mu _2 )$ which locally vanish, so the computational methods above do not apply. Is it possible to use the explicit formula for such cup products in \cite{MS} to compute these obstructions in practice?

\bibliography{bib_BK}

\newcommand{\etalchar}[1]{$^{#1}$}
\begin{thebibliography}{BDCKW18}

\bibitem[BCP97]{magma}
Wieb Bosma, John Cannon, and Catherine Playoust.
\newblock The {M}agma algebra system. {I}. {T}he user language.
\newblock {\em J. Symbolic Comput.}, 24(3-4):235--265, 1997.
\newblock Computational algebra and number theory (London, 1993).

\bibitem[BDCKW18]{BDCKW}
J.~S. Balakrishnan, I.~Dan-Cohen, M.~Kim, and S.~Wewers.
\newblock A non-abelian conjecture of {T}ate-{S}hafarevich type for hyperbolic curves.
\newblock {\em Math. Ann.}, 372(1-2):369--428, 2018.

\bibitem[Ber13]{berger}
Laurent Berger.
\newblock On {$p$}-adic {G}alois representations.
\newblock In {\em Elliptic curves, {H}ilbert modular forms and {G}alois deformations}, Adv. Courses Math. CRM Barcelona, pages 3--19. Birkh\"{a}user/Springer, Basel, 2013.

\bibitem[BK90]{BK}
S.~Bloch and K.~Kato.
\newblock {L}-functions and {T}amagawa numbers of motives, in {T}he {G}rothendieck {F}estschrift, {V}ol {I}.
\newblock pages 333--400. Birkh\"auser Boston, 1990.

\bibitem[BMS{\etalchar{+}}08]{BMSST}
Yann Bugeaud, Maurice Mignotte, Samir Siksek, Michael Stoll, and Szabolcs Tengely.
\newblock Integral points on hyperelliptic curves.
\newblock {\em Algebra Number Theory}, 2(8):859--885, 2008.

\bibitem[Bru03]{bruin}
Nils Bruin.
\newblock Chabauty methods using elliptic curves.
\newblock {\em J. Reine Angew. Math.}, 562:27--49, 2003.

\bibitem[BSS{\etalchar{+}}16]{LMFDB1}
Andrew~R. Booker, Jeroen Sijsling, Andrew~V. Sutherland, John Voight, and Dan Yasaki.
\newblock A database of genus-2 curves over the rational numbers.
\newblock {\em LMS J. Comput. Math.}, 19(suppl. A):235--254, 2016.

\bibitem[Dog23]{BKdescent1}
Netan Dogra.
\newblock $2$-descent for {B}loch--{K}ato {S}elmer groups {I}.
\newblock {\em preprint}, 2023.

\bibitem[FW99]{flynn-wetherell}
E.~Victor Flynn and Joseph~L. Wetherell.
\newblock Finding rational points on bielliptic genus 2 curves.
\newblock {\em Manuscripta Math.}, 100(4):519--533, 1999.

\bibitem[Has22]{hast}
Daniel~Rayor Hast.
\newblock Explicit two-cover descent for genus 2 curves.
\newblock {\em Res. Number Theory}, 8(4):Paper No. 67, 18, 2022.

\bibitem[HSV18]{HSV}
Wei Ho, Arul Shankar, and Ila Varma.
\newblock Odd degree number fields with odd class number.
\newblock {\em Duke Math. J.}, 167(5):995--1047, 2018.

\bibitem[MS03]{MS}
William~G. McCallum and Romyar~T. Sharifi.
\newblock A cup product in the {G}alois cohomology of number fields.
\newblock {\em Duke Math. J.}, 120(2):269--310, 2003.

\bibitem[NSW08]{NSW}
J\"{u}rgen Neukirch, Alexander Schmidt, and Kay Wingberg.
\newblock {\em Cohomology of number fields}, volume 323 of {\em Grundlehren der mathematischen Wissenschaften [Fundamental Principles of Mathematical Sciences]}.
\newblock Springer-Verlag, Berlin, second edition, 2008.

\bibitem[PS97]{PS97}
Bjorn Poonen and Edward~F. Schaefer.
\newblock Explicit descent for {J}acobians of cyclic covers of the projective line.
\newblock {\em J. Reine Angew. Math.}, 488:141--188, 1997.

\bibitem[Sti10]{stix2010trading}
J.~Stix.
\newblock Trading degree for dimension in the section conjecture: the non-abelian {S}hapiro lemma.
\newblock {\em Mathematical Journal of Okayama University}, 52(1), 2010.

\bibitem[Sti13]{stix2013correction}
J.~Stix.
\newblock Correction to: Trading degree for dimension in the section conjecture: The non-abelian shapiro lemma.
\newblock 2013.

\end{thebibliography}
\bibliographystyle{alpha}

\end{document}